\documentclass[a4paper,12pt,twoside]{article}

\usepackage{amsfonts}
\usepackage{amsmath}
\usepackage{amsopn}
\usepackage{amsthm}
\usepackage{amssymb}
\usepackage{graphicx}
\usepackage{rotating}
\usepackage[]{natbib}
\usepackage[english]{babel}
\usepackage[utf8]{inputenc}
\usepackage[colorlinks=true,citecolor=blue,linkcolor=blue]{hyperref}
\usepackage{xspace}
\usepackage{setspace}
\usepackage{dsfont}
\usepackage{booktabs}
\usepackage[T1]{fontenc}
\usepackage{anysize}
\usepackage{multirow}
\usepackage[labelfont=bf,labelsep=period]{caption}
\usepackage{floatrow}
\newfloatcommand{capbtabbox}{table}[][\FBwidth]
\marginsize{1.87cm}{1.87cm}{1.87cm}{1.87cm} 
\usepackage{fancyhdr}
\RequirePackage{txfonts}
\usepackage{times}

\newtheorem{Proposition}{Proposition}

\newtheorem{Theorem}{Theorem}
\newtheorem{Lemma}{Lemma}

\renewcommand{\d}{\mathrm{d}}
\DeclareMathOperator{\re}{\mathbb{R}}

\DeclareMathOperator{\yo}{\textsf{y}}
\newcommand{\abs}[1]{\left|#1\right|}
\newcommand{\E}{\mathbb{E}}
\newcommand{\ind}{\mathds{1}}
\newcommand{\ee}{\mathrm{e}}

\newcommand{\za}[1]{\stackrel{a}{#1}}
\newcommand{\zb}[1]{\stackrel{b}{#1}}
\newcommand{\zc}[1]{\stackrel{c}{#1}}
\newcommand{\zd}[1]{\stackrel{d}{#1}}
\newcommand{\ze}[1]{\stackrel{e}{#1}}

\allowdisplaybreaks

\begin{document}

\def\figureautorefname{Figure}
\def\tableautorefname{Table}
\def\algorithmautorefname{Algorithm}
\def\sectionautorefname{Section}
\def\subsectionautorefname{Section}
\def\Propositionautorefname{Proposition}
\def\Theoremautorefname{Theorem}
\def\Lemmaautorefname{Lemma}
\def\Assumptionautorefname{Assumption}
\def\Corollaryautorefname{Corollary}
\renewcommand*\footnoterule{}

\title{Theoretical properties of Bayesian Student-$t$ linear regression}

\author{Philippe Gagnon$^{1}$ and Yoshiko Hayashi$^2$}

\maketitle

\thispagestyle{empty}

\noindent $^{1}$Department of Mathematics and Statistics, Universit\'{e} de Montr\'{e}al, Canada.

\noindent $^{2}$Department of Economics, Osaka University of Economics, Japan.

\begin{abstract}
Bayesian Student-$t$ linear regression is a common robust alternative to the normal model, but its theoretical properties are not well understood. We aim to fill some gaps by providing analyses in two different asymptotic scenarios. The results allow to precisely characterize the trade-off between robustness and efficiency controlled through the degrees of freedom (at least asymptotically).
\end{abstract}


\section{Introduction}\label{sec:intro}

Let us assume that we have access to a data set of the form $(\mathbf{x}_i, y_i)_{i = 1}^n$, where $\mathbf{x}_1 := (x_{11}, \ldots, x_{1p})^T, \ldots, \mathbf{x}_n := (x_{n1}, \ldots, x_{np})^T \in \re^p$ are $n$ vectors with data points from $p$ covariates and $y_1, \ldots, y_n \in \re$ are $n$ observations of a dependent variable, with $n$ and $p$ being positive integers. The context is the following: one is interested in modelling the dependent variable using the covariates and a Bayesian linear regression model is assumed. We consider that $x_{11} = \ldots = x_{n1} = 1$ to introduce an intercept in the model. As typically done in linear regression, we treat (at least for now) the vectors $\mathbf{x}_1, \ldots, \mathbf{x}_n$ as fixed and known, i.e.\ not as realizations of random variables, contrarily to $y_1, \ldots, y_n$. The posterior distribution will thus be conditional on the latter only.

In linear regression, the random variables $Y_1, \ldots, Y_n$ are more precisely modelled as $Y_i = \mathbf{x}_i^T \boldsymbol\beta + \sigma\varepsilon_i$, $i = 1,\ldots, n$, where $\boldsymbol\beta := (\beta_1, \ldots, \beta_p)^T \in \re^p$ is the vector of regression coefficients, $\sigma > 0$ is a scale parameter, and $\varepsilon_1, \ldots, \varepsilon_n \in \re$ are random standardized errors. We assume that the $n + 2$ random variables $\varepsilon_1, \ldots, \varepsilon_n$, $\boldsymbol\beta$ and $\sigma$ are independent, implying that $\varepsilon_i \mid \boldsymbol\beta, \sigma \stackrel{\mathcal{D}}{=} \varepsilon_i \sim f, i = 1,\ldots, n$, where ``$\, \stackrel{\mathcal{D}}{=} \,$'' denotes an equality in distribution and $f$ is used to denote a probability density function (PDF). This independence assumption is common.

The most common choice of PDF $f$ is 
a standard normal density. This choice is however well known for yielding a model that lacks robustness against outliers; see, e.g., \cite{1968tiao119}, \cite{1984west431}, \cite{2009Pena2196} and \cite{gagnon2020} in which Bayesian robust alternatives to normal linear regression are proposed. The preferred alternative is the Student-$t$ linear regression, meaning that $f$ is replaced by a Student-$t$ density; hereafter, we write Student instead of Student-$t$ to simplify. That strategy dates back at least to \cite{1984west431}. Its scope of application is wide, ranging from modelling of mobility trends \citep{boonstra2021multilevel} to spatial modelling in ecological epidemiology \citep{congdon2017representing}.  The model can be automatically estimated using the probabilistic programming language \textit{Stan} \citep{carpenter2017stan}.

Even though the Student linear regression is the preferred Bayesian robust alternative to normal linear regression, not a lot is known about the theoretical properties of that model. Robustness properties of simple special cases are understood, such as the location model \citep{o1979outlier} and the location--scale model \citep{andrade2011bayesian}, corresponding to linear regression models with only an intercept and with a known $\sigma$ in the former model, but not beyond that. With this paper, we aim to improve understanding of the general model by studying some of its theoretical properties. From now on, we thus consider that $f$ is a Student density with $\gamma > 0$ degrees of freedom, a parameter that is fixed and chosen by the user. This parameter can be considered as being unknown, as in \cite{fernandez1999multivariate}, \cite{fonseca2008objective}, and \cite{he2021objective}, in which considerations specific to this situation are addressed and in which the proposed approaches are theoretically and empirically studied. We here focus on the situation where $\gamma$ is fixed as it is also an important and practical situation. To simplify, we will consider that $\gamma$ is a positive integer. Again to simplify, we will consider that all covariates are continuous; the theoretical results presented in the next sections hold even when this is not the case, but under more technical assumptions.

In \autoref{sec:proper}, we present a condition on the prior density which, together with the framework presented above, guarantee that the posterior distribution is proper if $n > p + 1$, this latter condition being similar to what is required under the frequentist paradigm to perform inference. The condition on the prior density is weak; for instance, it is satisfied by the improper Jeffreys prior. Even though a proper posterior distribution is required to perform inference under the Bayesian paradigm, theoretical guarantees that it is the case are scarce.

We next turn in \autoref{sec:robustness} to a characterization of the robustness of the model under an asymptotic scenario where outliers are considered to be further and further away from the bulk of the data. We prove that the posterior distribution converges towards one for which the PDF terms of the outlying data points in the original posterior density are each replaced by $\sigma^\gamma$ in the limiting one, but everything else remains the same. The term $\sigma^\gamma$ represents a trace asymptotically left by each outlier, which makes them \emph{partially} rejected. The trace increases the limiting posterior variability of all coefficients. The increase is seen to be more or less significant depending on several factors, as explained and shown in \autoref{sec:robustness}. The increase is small for certain combinations of those factors, but the degrees of freedom are of crucial importance; larger degrees of freedom imply larger variability increases. Also, larger degrees of freedom imply that a greater distance between the outliers and the bulk of the data is required for the former to be (partially) rejected, which translates into a stronger influence on inference of outliers that are not far enough to be (partially) rejected.

We finish our study of Bayesian Student linear regression in \autoref{sec:efficiency} with an analysis of its efficiency in a large-sample asymptotic regime where the true generating process is the normal linear regression in order to compare a Bayesian Student estimator with the ordinary-least-squares (OLS) one when the latter is the benchmark. We prove that the efficiency of the regression-coefficient estimator is comparable: the Bayesian Student estimator has an asymptotic variance (where the randomness here comes from the data) which is proportional to that of the OLS estimator, with a factor of proportionality that is greater than 1 but that converges to 1 as $\gamma$ increases.

With the findings presented in Sections \ref{sec:robustness} and \ref{sec:efficiency} in hand, one is able to precisely measure the impact of one's choice of value of $\gamma$ (at least asymptotically), with smaller values yielding greater robustness and larger values producing more efficient estimators. Our findings suggest that degrees of freedom around 4 are generally suitable, supporting previous evidence. Note that the proofs of all theoretical results are deferred to \autoref{sec_proofs}.

\section{Properness}\label{sec:proper}

Under the linear-regression framework described in \autoref{sec:intro}, the posterior density is such that
\begin{align}\label{eqn:post}
 \pi(\boldsymbol\beta, \sigma \mid \mathbf{y}) := \pi(\boldsymbol\beta, \sigma) \left[\prod_{i = 1}^n (1 / \sigma) f((y_i - \mathbf{x}_i^T \boldsymbol\beta) / \sigma)\right] \Bigg/ m(\mathbf{y}), \quad \boldsymbol\beta \in \re^p, \sigma > 0,
\end{align}
where $\mathbf{y} := (y_1, \ldots, y_n)^T$, $\pi(\, \cdot \,, \cdot \,)$ is the prior density and
\begin{align*}
 m(\mathbf{y}) :=  \int_{\re^p} \int_0^\infty \pi(\boldsymbol\beta, \sigma) \left[\prod_{i = 1}^n (1 / \sigma) f((y_i - \mathbf{x}_i^T \boldsymbol\beta) / \sigma)\right] \, \d\sigma \, \d\boldsymbol\beta,
\end{align*}
if $m(\mathbf{y}) < \infty$, a situation where the posterior distribution is proper and thus well defined. It is crucial to identify conditions under which this posterior distribution is proper given that any Bayesian analysis based on the Student linear regression rests on this distribution (and the fact that it is proper). We present in \autoref{prop:proper} sufficient conditions allowing the use of improper prior distributions. The conditions cover the Jeffreys prior, i.e.\ $\pi(\boldsymbol\beta, \sigma) \propto 1 / \sigma$, and $\pi(\boldsymbol\beta, \sigma) \propto 1$. The conditions require $n > p + 1$. Of course, other sets of conditions requiring less observations are possible, for instance, when assuming proper prior distributions.
\begin{Proposition}\label{prop:proper}
Assume that $\pi(\boldsymbol\beta, \sigma) \leq \max(C, C / \sigma)$ for all $\boldsymbol\beta, \sigma$, where $C$ is a positive constant. If $n > p + 1$, then the posterior distribution is proper.
\end{Proposition}
One can establish that moments of order $M$ exist if the condition $n > p + 1$ is replaced by $n > p + 1 + M$. In variable selection, when the joint posterior of the models and parameters is considered, this joint posterior is proper if the prior distributions of the parameters of all models satisfy the upper bound in \autoref{prop:proper} and if $n > p_{\max} + 1$, where $p_{\max} < \infty$ is the number of covariates in the \text{complete} model (the model with all covariates). In the analogous situation where $\gamma$ is considered unknown, the joint posterior of $\gamma$ and the parameters is proper if we assume that $\gamma$ takes values in a finite subset of the positive integers and the prior distributions of the parameters of all models (resulting from different values for $\gamma$) satisfy the upper bound in \autoref{prop:proper}.

\section{Robustness}\label{sec:robustness}

In this section, we state a result characterizing the robustness of the Student linear regression against outliers. 
An outlier here is defined as a data point $(\mathbf{x}_i, y_i)$ with an extreme error $y_i - \mathbf{x}_i^T \boldsymbol\beta$, for $\boldsymbol\beta$ belonging to the set of probable values according to the bulk of the data. An error can be extreme because, for a given $\mathbf{x}_i$, the value of $y_i$ makes it extreme or because, for a given $y_i$, the value of $\mathbf{x}_i$ makes it extreme. We mathematically represent such extreme situations by considering an asymptotic scenario where the outliers move away from the bulk of the data along particular paths (see \autoref{fig:paths}). More precisely, we consider that the outlying data points $(\mathbf{x}_i, y_i)$ are such that $y_i \rightarrow \pm \infty$, with $\mathbf{x}_i$ being kept fixed (but perhaps extreme). Our result states that, for the outlying data points with fixed $\mathbf{x}_i$, there exist $y_i$ values such that the posterior distribution is similar to one for which the PDF terms of the outliers are each replaced by $\sigma^\gamma$.

   \begin{figure}[ht]
 \centering
   \includegraphics[width=0.40\textwidth]{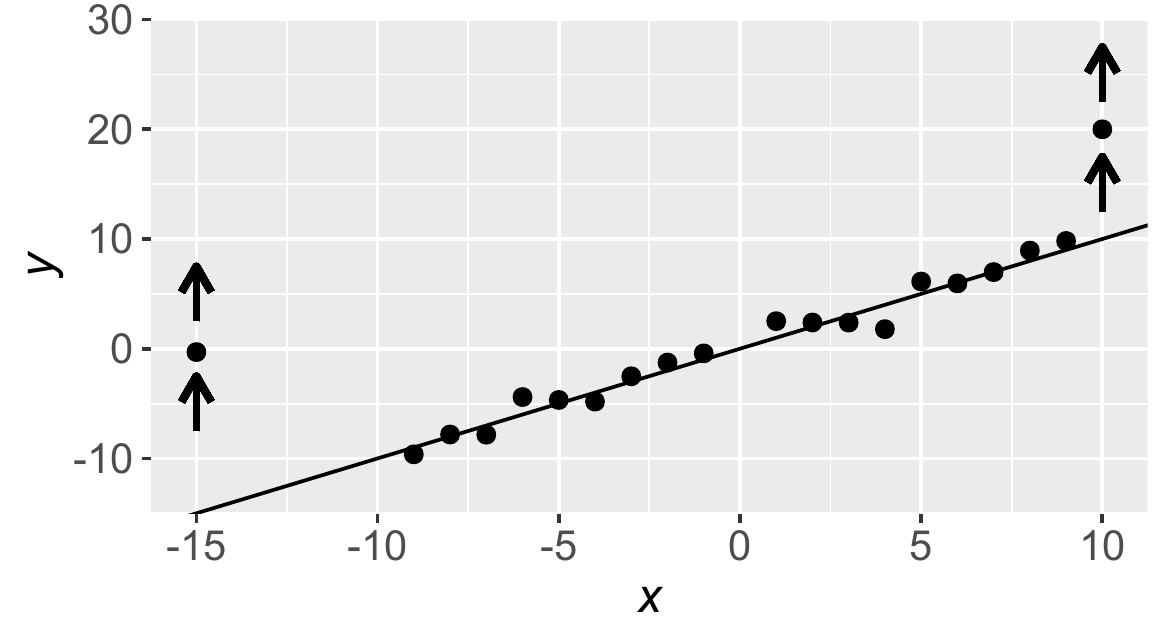}
  \vspace{-3mm}
\caption{Data with two outliers: one can be seen as having an extreme $y_i$, while the other as having an extreme $\mathbf{x}_i$}\label{fig:paths}
\end{figure}
\normalsize

We consider that each outlier goes to $\pm \infty$ at its own specific rate, to the extend that the ratio of two outliers is bounded: $y_i = a_i + b_i \omega$, for $i = 1, \dots, n$, where $a_i, b_i \in \re$ are constants such that $b_i = 0$ if the data point is a non-outlier and $b_i \neq 0$ if it is an outlier, and then, we let $\omega \rightarrow \infty$. We define the index set of outlying data points by: $\text{O} := \{i: b_i \neq 0\}$. The index set of non-outlying data points is thus given by: $\text{O}^\mathsf{c} = \{1, \ldots, n\} \setminus \text{O}$. We also define the set of non-outlying observations: $\mathbf{y}_{\text{O}^\mathsf{c}} := \{y_i: i \in \text{O}^\mathsf{c}\}$.

Central to the characterization of the robustness of the Student linear regression is the limiting behaviour of the Student PDF evaluated at an outlying point: for any $i \in \text{O}$ and fixed $(\boldsymbol\beta, \sigma)$,
\begin{align}\label{eqn:limit_PDF}
 \lim_{\omega \rightarrow \infty} \frac{(1 / \sigma) f((y_i - \mathbf{x}_i^T \boldsymbol\beta) / \sigma)}{f(y_i)} = \lim_{\omega \rightarrow \infty} \frac{1}{\sigma} \left(\frac{\gamma + y_i^2}{\gamma + (y_i - \mathbf{x}_i^T \boldsymbol\beta)^2 / \sigma^2}\right)^{\frac{\gamma + 1}{2}} = \sigma^\gamma.
\end{align}
This result is a consequence of the polynomial behaviour of the Student PDF tails which makes them regularly varying; for details on regularly varying function, see \cite{resnick2007heavy}. The result suggests that the PDF term of an outlier in the posterior density behaves in the limit like $\sigma^\gamma f(y_i) \propto \sigma^\gamma$ (recall \eqref{eqn:post}). This is confirmed in a theorem below.

An outlier can be seen as a source of information that is in conflict with others. The sources with which it is in conflict represent, among others, the non-outliers. Here, we consider that the prior distribution is not in conflict with the non-outliers. A line of research called \textit{resolution of conflict} studies how conflicting sources of information are dealt with by Bayesian models. This line of research was started by \cite{de1961bayesian} with a first analysis in \cite{lindley1968choice}, followed by an introduction of a formal theory in \cite{dawid1973posterior}, \cite{hill1974coherence} and \cite{o1979outlier}. For a review of Bayesian heavy-tailed models and conflict resolution, see \cite{o2012bayesian}. In the latter paper, it is noted that there exists a gap between the models formally covered by the theory of conflict resolution and models commonly used in practice. The present paper contributes to the expansion of the theory of conflict resolution by covering a model used in practice, namely the Student linear regression.

In the presence of conflicting sources of information and when using the Student linear regression, conflicting information is partially rejected as a trace remains, i.e.\ $\sigma^\gamma$. Ideally, conflicting information is wholly rejected as its source becomes increasingly remote, which translates into a PDF which behaves asymptotically like $f(y_i) \propto 1$ \citep{1984west431}. The Student model is thus said to be \textit{partially robust}.  Recently, research on resolution of conflict in linear-regression frameworks have shown that in order to attain whole robustness, it is required to assume that the error PDF has heavier density tails than those of the Student PDF \citep{DesGag2019, gagnon2020, gagnon2020PCR, hamura2020log}; in \cite{gagnon2020}, it is proved that it is sufficient to assume that the error PDF has tails which are log-regularly varying, a concept introduced in \cite{desgagne2015robustness}.

The trace $\sigma^\gamma$ left asymptotically by each outlier has an impact on the limiting posterior variability of all coefficients which is seen to be more or less significant depending on the sample size, the number of outliers, the number of covariates in the model and the degrees of freedom. When the sample size is large relatively to the number of outliers, the number of covariates and the degrees of freedom, the impact is small. The regularly-varying Student tails however make the convergence to the limiting posterior distribution slower, comparatively to other alternatives with heavier tails such as log-regularly tails, implying a slower (in this case, partial) resolution of conflict. In \autoref{fig:y_20} and \hyperref[table:1]{Table 1}, we present numerical results to make some of that concrete.\footnote{The code to reproduce all numerical results is available online; see ancillary files on \href{https://arxiv.org/abs/2204.02299}{arXiv:2204.02299}.} They are based on an analysis of a simulated data set with $n = 20$, $p = 2$, $(x_{12}, \ldots, x_{n2}) = (1, 2, \ldots, n)$, and where $y_1, \ldots, y_n$ were sampled using intercept and slope coefficients both equal to 1, an error scaling of 1 and errors sampled independently from the standard normal distribution; we then gradually increase the value of $y_n$. We observe the impact on the posterior mean of $\beta_2$ for different values of $\gamma$ in \autoref{fig:y_20}. In \hyperref[table:1]{Table 1}, we show the difference between the posterior means and standard deviations (SDs) of $\beta_2$ based on the limiting posterior distribution (in which $(1 / \sigma) f((y_n - \mathbf{x}_n^T \boldsymbol\beta) / \sigma)$ is replaced by $\sigma^\gamma$) and the posterior distribution without $y_n$ (corresponding to the limiting distribution as if the approach were wholly robust). From these results, we observe that models with degrees of freedom around 4 are almost as robust as the model with $\gamma = 1$, with a limited impact on the posterior coefficient variability due to their partial robustness, while being closer to normal linear regression, which is an advantage when non-outlying data points are normally distributed. The posterior means and SDs were computed using Hamiltonian Monte Carlo \citep{Duane1987} and Markov-chain samples of size 10,000,000 (see \autoref{sec:details_exp} for details). For an extensive simulation study, we refer the reader to \cite{gagnon2020}.

   \begin{figure}[ht]
 \begin{floatrow}
\ffigbox{%
  \includegraphics[width=0.50\textwidth]{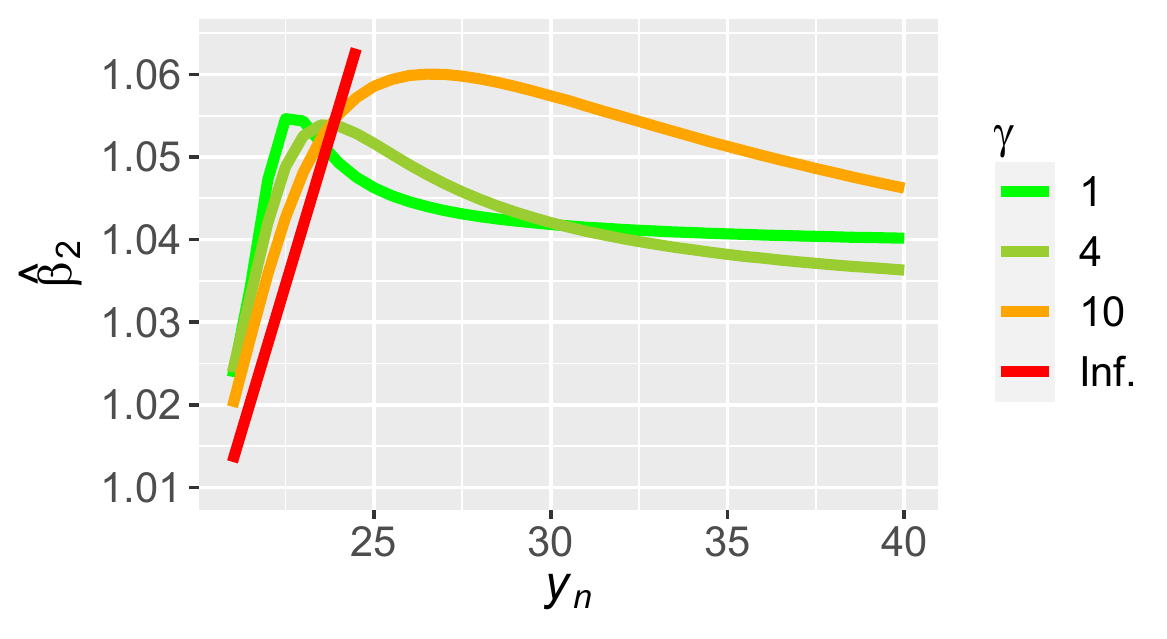}%
}{%
\vspace{-1mm}
  \caption{\small Impact on the posterior mean of $\beta_2$ as $y_{n}$ varies for different values of $\gamma$ (``Inf.'' represents the standard normal)}\label{fig:y_20}
}
\capbtabbox{%
\small
\begin{tabular}{l rrr}
\toprule
  &  & $\hat{\beta}_2$ & SD \cr
\midrule
\multirow{2}{*}{$\gamma = 1$} & limiting distribution & $1.039$ & $0.037$ \cr
                                                    & without outlier & $1.039$ & $0.035$ \cr
\midrule
\multirow{2}{*}{$\gamma = 4$} & limiting distribution & $1.030$ & $0.045$ \cr
                                                    & without outlier & $1.034$ & $0.038$ \cr
\midrule
\multirow{2}{*}{$\gamma = 10$} & limiting distribution & $1.022$ & $0.072$ \cr
                                                    & without outlier & $1.027$ & $0.040$ \cr
\midrule
$\gamma = +\infty$ & without outlier & $1.018$ & $0.043$ \cr
\bottomrule
\end{tabular}
}{%
  \caption{\small Posterior means and SDs of $\beta_2$ based on the limiting posterior distribution as $y_{n} \rightarrow \infty$ and the posterior distribution without this observation, for different values of $\gamma$ ($\gamma = +\infty$ represents the standard normal)}\label{table:1}
}
\end{floatrow}
\end{figure}
\normalsize


The theoretical result that we demonstrate is a convergence of the posterior distribution towards $\pi(\, \cdot \,, \, \cdot \mid \mathbf{y}_{\text{O}^\mathsf{c}})$, which has a density defined as follows:
\[
 \pi(\boldsymbol\beta, \sigma \mid \mathbf{y}_{\text{O}^\mathsf{c}}) := \pi(\boldsymbol\beta, \sigma) \, \sigma^{|\text{O}| \gamma} \left[\prod_{i \in \text{O}^\mathsf{c}} (1 / \sigma) f((y_i - \mathbf{x}_i^T \boldsymbol\beta) / \sigma)\right] \Bigg/ m(\mathbf{y}_{\text{O}^\mathsf{c}}), \quad \boldsymbol\beta \in \re^p, \sigma > 0,
\]
where $|\text{O}|$ is the cardinality of the set $\text{O}$, i.e.\ the number of outliers, and
\begin{align*}
 m(\mathbf{y}_{\text{O}^\mathsf{c}}) :=  \int_{\re^p} \int_0^\infty \pi(\boldsymbol\beta, \sigma) \, \sigma^{|\text{O}| \gamma} \left[\prod_{i \in \text{O}^\mathsf{c}} (1 / \sigma) f((y_i - \mathbf{x}_i^T \boldsymbol\beta) / \sigma)\right] \, \d\sigma \, \d\boldsymbol\beta.
\end{align*}
To prove the result, we essentially need to prove that we can interchange the limit and the integral in
\[
 \lim_{\omega \rightarrow \infty} \int_{\re^p} \int_0^\infty \pi(\boldsymbol\beta, \sigma) \left[\prod_{i = 1}^n (1 / \sigma) f((y_i - \mathbf{x}_i^T \boldsymbol\beta) / \sigma)\right] \Bigg/ \prod_{i \in \mathrm{O}}f(y_i) \, \d\sigma \, \d\boldsymbol\beta = \lim_{\omega \rightarrow \infty}  \frac{m(\mathbf{y})}{\prod_{i \in \mathrm{O}}f(y_i)},
\]
which is the main difficulty of proving our asymptotic result.

To be able to establish such a result, we need a guarantee that $\pi(\, \cdot \,, \, \cdot \mid \mathbf{y}_{\text{O}^\mathsf{c}})$ is well defined. The following proposition provides conditions under which this guarantee exists.
\begin{Proposition}\label{prop:proper_limiting}
Assume that $\pi(\boldsymbol\beta, \sigma) \leq \max(C, C / \sigma)$ for all $\boldsymbol\beta, \sigma$, where $C$ is a positive constant. If $n - |\mathrm{O}|(\gamma + 1) = |\mathrm{O}^\mathsf{c}| - |\mathrm{O}|\gamma  > p + 1$, then $\pi(\, \cdot \,, \, \cdot \mid \mathbf{y}_{\mathrm{O}^\mathsf{c}})$ is proper.
\end{Proposition}

We now present the asymptotic result.

\begin{Theorem}\label{thm:1}
Assume that $\pi(\boldsymbol\beta, \sigma) \leq \max(C, C / \sigma)$ for all $\boldsymbol\beta, \sigma$, where $C$ is a positive constant. If $|\mathrm{O}^\mathsf{c}| \geq \max\{n/2+ p - 1/2, |\mathrm{O}|\gamma + p + 2\}$, then, as $\omega \rightarrow \infty$,
 \begin{description}
   \item[(a)] the asymptotic behaviour of the marginal distribution is: $m(\mathbf{y}) / \prod_{i \in \mathrm{O}}f(y_i)\rightarrow m(\mathbf{y}_{\mathrm{O}^\mathsf{c}})$;

     \item[(b)] the posterior density converges pointwise: for any $\boldsymbol\beta \in \re^p, \sigma > 0$, $\pi(\boldsymbol\beta, \sigma \mid \mathbf{y}) \rightarrow \pi(\boldsymbol\beta,\sigma\mid\mathbf{y}_{\mathrm{O}^\mathsf{c}})$;

\item[(c)] the posterior distribution converges: $\pi(\, \cdot \,, \, \cdot \mid \mathbf{y}) \rightarrow \pi(\, \cdot \,, \, \cdot \mid \mathbf{y}_{\mathrm{O}^\mathsf{c}})$.
\end{description}
\end{Theorem}

For the Student linear regression, once the prior distribution has been set such that $\pi(\boldsymbol\beta, \sigma) \leq \max(C, C / \sigma)$ for all $\boldsymbol\beta, \sigma$, it is seen that \autoref{thm:1} holds as long as the number of non-outliers is large enough. A sufficient number of non-outliers is $|\mathrm{O}^\mathsf{c}| \geq \max\{n/2+ p - 1/2, |\mathrm{O}|\gamma + p + 2\}$ which is equivalent to having an upper bound on the number of outliers: $\mathrm{O}$ must be such that $|\mathrm{O}| \leq n/2 - p + 1/2$ and $|\mathrm{O}| \leq (n - p - 2) / (\gamma + 1)$. This condition suggests that the breakdown point, generally defined as the proportion of outliers $|\mathrm{O}| / n$ that an estimator can handle, is the minimum of $1/2 - (p - 1/2) / n$ and $(n - p - 2) / (n(\gamma + 1))$, which is close to $1/ (\gamma + 1)$ if $n$ is large relatively to $p$. This is another argument in favour of small degrees of freedom, with the smaller trace left asymptotically by outliers in $\pi(\, \cdot \,, \, \cdot \mid \mathbf{y}_{\text{O}^\mathsf{c}})$.

In \autoref{thm:1}, Result (a) represents the centrepiece; it leads relatively easily to the other results of the theorem, but its demonstration requires considerable work. The convergence of the posterior density in Result (b) enables to state that the maximum a posteriori estimate is partially robust. Given that this estimate corresponds to the maximum likelihood estimate when the prior is proportional to 1, the frequentist estimate is, as a result, also partially robust. This allows establishing a connection between Bayesian and frequentist robustness. Result (c) indicates that any estimation of $\boldsymbol\beta$ and $\sigma$ based on posterior quantiles (e.g.\ using posterior medians or Bayesian credible intervals) is partially robust to outliers. Note that the convergence of posterior expectations holds if it is instead assumed that $|\mathrm{O}^\mathsf{c}| \geq \max\{n/2+ p - 1/2, |\mathrm{O}|\gamma + p + 2\} + 1$. All these
results characterize the limiting behaviour of a variety of Bayes estimators. Finally, we note that in variable selection, when the joint posterior of the models and parameters is considered, this joint posterior converges if the prior distributions of the parameters of all models satisfy the upper bound in \autoref{thm:1} and if $|\mathrm{O}^\mathsf{c}| \geq \max\{n/2+ p_{\max} - 1/2, |\mathrm{O}|\gamma + p_{\max} + 2\}$.

\section{Efficiency}\label{sec:efficiency}

To study the efficiency of estimators produced by Bayesian Student linear regression, we consider another asymptotic framework than that of the previous section: we perform an analysis under a large-sample regime $n \rightarrow \infty$. The analysis rests heavily on the theoretical results in \cite{bunke1998asymptotic} about asymptotic behaviour of Bayes estimates under possibly incorrect models. The analysis consists essentially in verifying the assumptions of Theorems 2.1, 2.2 and 4.1 in that paper which are stated in great generalities, thus requiring considerable work. The analysis allows to establish that the posterior distribution concentrates as $n \rightarrow \infty$ around \textit{pseudo-true parameters} and that \textit{pseudo-Bayes estimators} are strongly consistent and asymptotically normally distributed; these estimators are called pseudo-Bayes because they converge towards the pseudo-true parameters, the latter being values which make the model as close as possible to the true generating process in a sense specified below. To perform the analysis, we assume, contrarily to before, that each $\mathbf{x}_i$ is an observation from a random vector $\mathbf{X}_i$. Also, we assume that the true generating model is the following: $\mathbf{Z}_1 := (Y_1, \mathbf{X}_1), \ldots, \mathbf{Z}_n := (Y_n, \mathbf{X}_n)$ are independent and identically distributed (IID) random variables such that $\mathbf{X}_i \sim \mu_\mathbf{X}$ and $Y_i = \mathbf{X}_i^T \boldsymbol\beta_0 + \sigma_0\varepsilon_i$ with $\boldsymbol\beta_0 \in \re^p$ being the true (fixed) coefficient vector, $\sigma_0 >0$ being the true (fixed) scale parameter, $\varepsilon_1, \ldots, \varepsilon_n \sim g$ being $n$ IID random variables, $g$ being a PDF and $\mu_\mathbf{X}$ being a probability measure. Given that measuring the efficiency of estimators in linear-regression frameworks is often done in the situation where normal linear regression is the gold standard to compare to the latter, we further assume that $g$ is the PDF of the standard normal distribution.

We will consider that all the components in $\mathbf{X}_i$ are random, except the first one (the intercept). The distribution $\mu_\mathbf{X}$ is thus that of all the components in $\mathbf{X}_i$, except the first one. We will continue to write $\mu_\mathbf{X}$ to simplify. This distribution thus has a density with respect to Lebesgue measure (because the covariates are all assumed to be continuous); this density is denoted by $f_\mathbf{X}$.

In our Bayesian model, we proceed as before, but we need to consider that the posterior distribution is a conditional distribution given both $y_1, \ldots, y_n$ and $\mathbf{x}_1, \ldots, \mathbf{x}_n$. We assume in our model that $\mathbf{Z}_1, \ldots, \mathbf{Z}_n$ are such that $Y_i = \mathbf{X}_i^T \boldsymbol\beta + \sigma\varepsilon_i$ for all $i$. We can show that, by assuming that $\mathbf{X}_1, \ldots, \mathbf{X}_n, \varepsilon_1, \ldots, \varepsilon_n, \boldsymbol\beta, \sigma$ are independent random variables, the posterior distribution does not in fact depend on the assumed distribution of $\mathbf{X}_i$; we can thus assume in our model that $\mathbf{X}_i \sim\mu_\mathbf{X}$, i.e.\ that each $\mathbf{X}_i$ has the correct distribution. We have, similarly to before, that $\varepsilon_i \mid \mathbf{X}_i, \boldsymbol\beta, \sigma \stackrel{\mathcal{D}}{=} \varepsilon_i \sim f, i = 1,\ldots, n$; it is thus assumed that each error has a Student distribution. The model is thus misspecified and the pseudo-true parameters are the \textit{closest} (in some sense) to $(\boldsymbol\beta_0, \sigma_0)$ when using a Student linear regression instead of a normal one. The pseudo-true parameters are characterized precisely below.

Let us define the conditional PDF of $Y_i$ given $\mathbf{X}_i = \mathbf{x}_i$ indexed by fixed $(\boldsymbol\beta, \sigma)$ under the Student model to be $p_{(\boldsymbol\beta, \sigma)}( \, \cdot \mid \mathbf{x}_i)$.
Let us also define the conditional PDF of $Y_i$ given $\mathbf{X}_i = \mathbf{x}_i$ under the true model to be $p_{0}( \, \cdot \mid \mathbf{x}_i)$.
The first step in the analysis is the identification of pseudo-true parameters $(\boldsymbol\beta^*, \sigma^*)$ which minimizes the divergence
\[
 K(\boldsymbol\beta, \sigma) := \E\left[\log\frac{p_{0}(Y \mid \mathbf{X}) \, f_\mathbf{X}(\mathbf{X})}{p_{(\boldsymbol\beta, \sigma)}(Y \mid \mathbf{X}) \, f_\mathbf{X}(\mathbf{X})}\right] = \E\left[\log\frac{p_{0}(Y \mid \mathbf{X})}{p_{(\boldsymbol\beta, \sigma)}(Y \mid \mathbf{X})}\right],
\]
where the expectations during this analysis are always taken under the true model, i.e.\ with $\mathbf{X} \sim \mu_\mathbf{X}$ and $Y \mid \mathbf{X} \sim p_{0}( \, \cdot \mid \mathbf{X})$. We now state a result which allows to characterize the pseudo-true parameters and presents conditions under which the posterior distribution concentrates around the pseudo-true parameters as $n \rightarrow \infty$ and the pseudo-Bayes estimators are strongly consistent and asymptotically normally distributed. The pseudo-Bayes estimators $(\hat{\boldsymbol\beta}_n, \hat{\sigma}_n)$ that we consider are posterior means; they are thus pseudo-Bayes with respect to the quadratic loss. Other pseudo-Bayes estimators can be considered under more technical assumptions. Note that the randomness in the posterior distribution and $(\hat{\boldsymbol\beta}_n, \hat{\sigma}_n)$ comes from the data points $\mathbf{Z}_1, \ldots, \mathbf{Z}_n$. 

\begin{Theorem}\label{thm:2}
 Consider the framework described in this section.
 \begin{description}
   \item[(a)] For any $\boldsymbol\beta_0, \sigma_0$ and $\gamma$, the function $K(\boldsymbol\beta, \sigma)$ has a unique minimum that is attained at $(\boldsymbol\beta_0, \sigma^*)$, where $\sigma^*$ is the (unique) solution to
\begin{align}\label{eqn:Thm_2_a}
        (\gamma + 1) \left[1 - \sqrt{2 \pi (\sigma / \sigma_0)^2 \gamma} \, \exp\left((\sigma / \sigma_0)^2 \gamma / 2\right) \Phi\left(-\sqrt{(\sigma / \sigma_0)^2 \gamma}\right)\right] - 1 = 0,
       \end{align}
       where $\Phi$ is the cumulative distribution function of the standard normal distribution.

  \item[(b)] Assume that $f_\mathbf{X}$ is bounded, $\E[\|\mathbf{X}\|^{4 (p + 2)}] < \infty$ and $\|\boldsymbol\beta\|^{1/32}\E[|\mathbf{X}^T \boldsymbol\beta|^{-1/32}] < \infty$ for any $\boldsymbol\beta$, where $\| \, \cdot \, \|$ is the Euclidean norm. Assume that the prior density $\pi(\, \cdot \,, \cdot \,)$ is strictly positive and such that $\pi(\boldsymbol\beta, \sigma) \leq \max(C, C / \sigma)$ for all $\boldsymbol\beta, \sigma$. The posterior distribution concentrates around $(\boldsymbol\beta_0, \sigma^*)$ as $n \rightarrow \infty$ with probability 1 and the posterior means are strongly consistent:
      \[
       \lim_{n \rightarrow \infty} \, (\hat{\boldsymbol\beta}_n, \hat{\sigma}_n) = (\boldsymbol\beta_0, \sigma^*) \quad \text{with probability 1}.
      \]

  \item[(c)] Further assume that $\E[\mathbf{X} \mathbf{X}^T]$ is a positive-definite matrix and that the prior density $\pi(\, \cdot \,, \cdot \,)$ is continuous. The posterior means are asymptotically normally distributed, i.e.\ we have the following convergence in distribution:
      \[
       \sqrt{n}((\hat{\boldsymbol\beta}_n, \hat{\sigma}_n) - (\boldsymbol\beta_0, \sigma^*)) \rightarrow \mathcal{N}(\mathbf{0}, \mathbf{M}^{-1} \mathbf{I} \mathbf{M}^{-1}),
      \]
      where
      \[
 \mathbf{M} := -\E[l''(\mathbf{Z}, (\boldsymbol\beta_0, \sigma^*))], \quad \mathbf{I} := \E[l'(\mathbf{Z}, (\boldsymbol\beta_0, \sigma^*)) l'(\mathbf{Z}, (\boldsymbol\beta_0, \sigma^*))^T],
\]
$l(\mathbf{z}, (\boldsymbol\beta, \sigma)) := \log p_{(\boldsymbol\beta, \sigma)}(y \mid \mathbf{x})$ and $l'(\mathbf{z}, (\boldsymbol\beta, \sigma))$ and $l''(\mathbf{z}, (\boldsymbol\beta, \sigma))$ are the vector of first derivatives and matrix of second derivatives with respect to $(\boldsymbol\beta, \sigma)$, respectively; see \autoref{sec_proofs} for the detailed expressions.
 \end{description}
\end{Theorem}

We now make a few remarks about \autoref{thm:2}. Firstly, it establishes that even if the Bayesian Student model is misspecified, the latter allows to retrieve the true regression coefficients $\boldsymbol\beta_0$ as $n \rightarrow \infty$, but the scale-parameter value around which the posterior distribution concentrates differs from $\sigma_0$ by a factor; see \autoref{fig:sigma_star}.

   \begin{figure}[ht]
 \centering
   \includegraphics[width=0.40\textwidth]{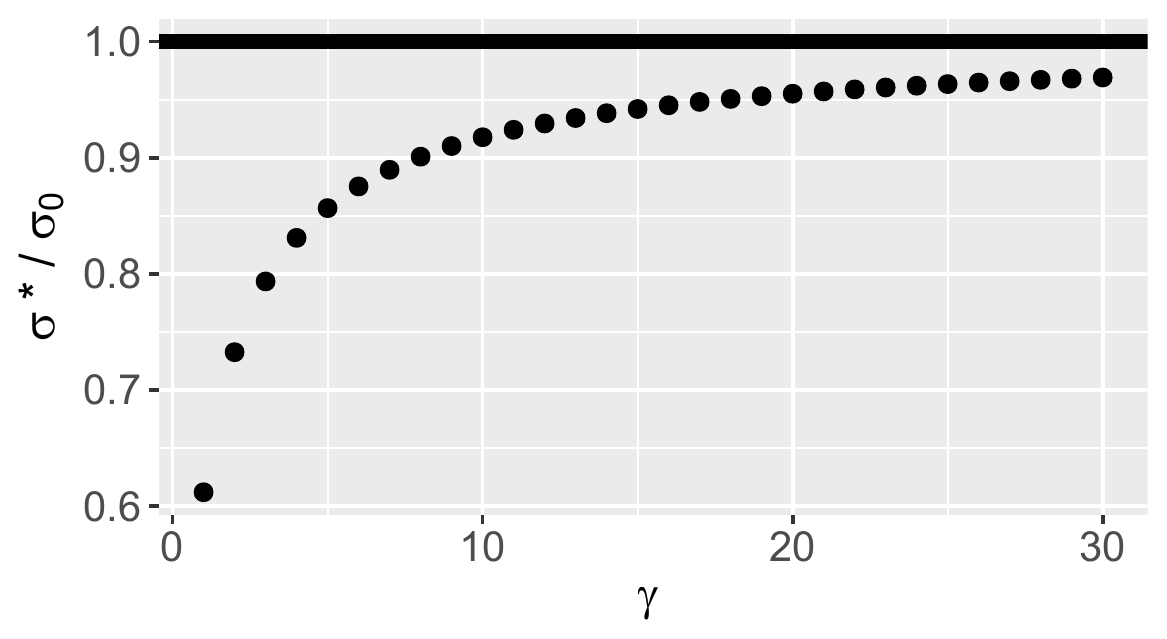}
  \vspace{-3mm}
\caption{$\sigma^* / \sigma_0$ as a function of $\gamma$; the values have been found by numerically solving \eqref{eqn:Thm_2_a}}\label{fig:sigma_star}
\end{figure}
\normalsize

Secondly, the assumptions made to derive the results are mainly about the distribution of $\mathbf{X}$ and are mainly regularity conditions. There is one that is more technical:  $\|\boldsymbol\beta\|^{1/32}\E[|\mathbf{X}^T \boldsymbol\beta|^{-1/32}] < \infty$ for any $\boldsymbol\beta$. It can be seen to be satisfied when the elements in $\mathbf{X}$ (except the intercept) are independent random variables each having a standard normal distribution. Indeed, in this case, for any $\boldsymbol\beta$,
\[
 \|\boldsymbol\beta\|^{1/32}\E[|\mathbf{X}^T \boldsymbol\beta|^{-1/32}] = \E\left[\left|\frac{\beta_1}{\|\boldsymbol\beta\|} + \frac{\sqrt{\beta_2^2 + \ldots + \beta_p^2}}{\|\boldsymbol\beta\|} \, Z\right|^{-1/32}\right] < \infty, \quad Z \sim \mathcal{N}(0, 1).
\]

We finish with a remark about the efficiency of the Bayesian estimators under the misspecified Student-linear-regression model. When one wants to compare the efficiency of an estimator to that of the benchmark estimator, a measure that is often used is the asymptotic relative efficiency. It consists in comparing the variances in the asymptotic distributions (which are often normal) after standardizing the estimators in the same manner. The benchmark estimator here is that associated with the normal linear regression when the latter is the true model, meaning the OLS estimator. The efficiency of the scale estimators cannot be compared because the posterior mean $\hat{\sigma}_n$ under the Student model does not converge to $\sigma_0$ as $n \rightarrow \infty$; we will thus focus on comparing the efficiency of the coefficient estimators, which is of main interest. Under the framework presented in this section, we have the following convergence in distribution for the OLS estimator:
\[
 \sqrt{n}(\hat{\boldsymbol\beta}_n^{\text{OLS}} - \boldsymbol\beta_0) \rightarrow \mathcal{N}(\mathbf{0}, \sigma_0^2 \, \E[\mathbf{X} \mathbf{X}^T]^{-1}).
\]
When focusing on the asymptotic behaviour of $\hat{\boldsymbol\beta}_n$ under the Student model, the asymptotic variance is relatively simple:
\[
 \sqrt{n}(\hat{\boldsymbol\beta}_n - \boldsymbol\beta_0) \rightarrow \mathcal{N}(\mathbf{0}, \sigma_0^2 \, \E[\mathbf{X} \mathbf{X}^T]^{-1} \phi(\gamma)),
\]
where
\[
 \phi(\gamma) := \E\left[\frac{Z^2}{\left(1+\frac{Z^2}{(\sigma^* / \sigma_0)^2 \gamma}\right)^2}\right] \E\left[\frac{1-\frac{Z^2}{(\sigma^* / \sigma_0)^2 \gamma}}{\left(1+\frac{Z^2}{(\sigma^* / \sigma_0)^2 \gamma}\right)^2}\right]^{-2}, \quad Z \sim \mathcal{N}(0, 1).
\]
We wrote $\phi$ as a function of $\gamma$ only, because the ratio $(\sigma^* / \sigma_0)$ is itself a function of $\gamma$ (as seen in \autoref{thm:2} and \autoref{fig:sigma_star}). \autoref{fig:phi} presents values of $\phi$ for several values of $\gamma$; the expectations in $\phi$ have been numerically evaluated using adaptive quadrature. With this figure and the results presented above, we are able to conclude that the Bayesian coefficient estimator under the misspecified Student-linear-regression model is comparable in terms of efficiency to the OLS coefficient estimator, in the sense that their asymptotic variances are the same, up to a factor $\phi(\gamma)$. This factor is seen to converge to 1 as $\gamma$ increases, which is expected given that the Student distribution resembles more and more the normal one as $\gamma$ increases. With $\gamma = 4$, the factor by which the asymptotic variance of the Bayesian coefficient estimator is inflated is around 10\%, which seems a reasonable price to pay for robustness in the presence of outliers.

 \begin{figure}[ht]
 \centering
   \includegraphics[width=0.40\textwidth]{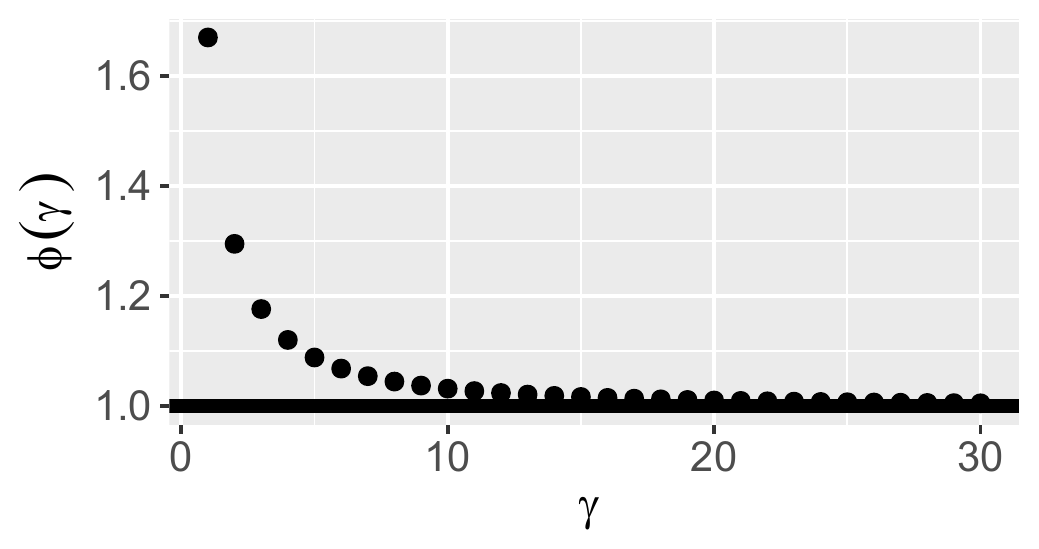}
  \vspace{-3mm}
\caption{ $\phi$ as a function of $\gamma$}\label{fig:phi}
\end{figure}
\normalsize

\section{Acknowledgements}

Philippe Gagnon acknowledges support from NSERC (Natural Sciences and Engineering Research Council of Canada) and FRQNT (Fonds de recherche du Qu\'{e}bec -- Nature et technologies). Also, the authors thank two anonymous referees for helpful suggestions that led to an improved manuscript.

\bibliographystyle{rss}
\bibliography{references}

\appendix

\section{Proofs}\label{sec_proofs}

\begin{proof}[Proof of \autoref{prop:proper}]
 The result follows from Proposition 2.1 in \cite{gagnon2020}. It can indeed be readily verified that the assumptions required to apply Proposition 2.1 are verified.
\end{proof}

\begin{proof}[Proof of \autoref{prop:proper_limiting}]
To prove this result, we show that $m(\mathbf{y}_{\text{O}^\mathsf{c}})$ is finite. We proceed as follows: we find an upper bound for the integral, which will turn out to be a special case of the integral that is shown to be finite in the proof of  Proposition 2.1 in \cite{gagnon2020}. Without loss of generality, we consider that $\text{O}^\mathsf{c} = \{1, \ldots, | \text{O}^\mathsf{c}|\}$. We have
 \begin{align*}
  m(\mathbf{y}_{\text{O}^\mathsf{c}}) &= \int_{\re^p} \int_0^\infty \pi(\boldsymbol\beta, \sigma) \, \sigma^{|\text{O}| \gamma} \left[\prod_{i = 1}^{| \text{O}^\mathsf{c}|} (1 / \sigma) f((y_i - \mathbf{x}_i^T \boldsymbol\beta) / \sigma)\right] \, \d\sigma \, \d\boldsymbol\beta \cr
  &\leq C^{|\text{O}| \gamma}\int_{\re^p} \int_0^\infty \pi(\boldsymbol\beta, \sigma) \left[\prod_{i = 1}^{|\text{O}^\mathsf{c}| - |\text{O}| \gamma} (1 / \sigma) f((y_i - \mathbf{x}_i^T \boldsymbol\beta) / \sigma)\right] \, \d\sigma \, \d\boldsymbol\beta,
 \end{align*}
 using that we can choose $C$ such that $f \leq C$. It can be readily verified that this integral is a special case of the integral in the proof of Proposition 2.1 in \cite{gagnon2020} that is shown to be finite when the number of terms in the product $|\text{O}^\mathsf{c}| - |\text{O}| \gamma = n - |\text{O}| - |\text{O}| \gamma > p + 1$.
\end{proof}

We now present a lemma that will be useful in the next proofs.

\begin{Lemma}\label{lemma:bound_student}
 For any $\nu \geq 1$, we have
 \[
  \frac{f(z / \nu)}{\nu^{\gamma + 1} f(z)} \leq 1, \quad \text{for all $z \in \re$}.
 \]
\end{Lemma}

\begin{proof}[Proof of \autoref{lemma:bound_student}]
 Given that $\nu \geq 1$, we have
 \[
  \frac{1}{\nu^2} + \frac{z^2 /  \nu^2}{\gamma} \leq 1 + \frac{z^2 /  \nu^2}{\gamma},
 \]
 implying that
 \begin{align*}
  \frac{f(z / \nu)}{\nu^{\gamma + 1} f(z)} = \frac{\left(1 + \frac{z^2 / \nu^2}{\gamma}\right)^{-\frac{\gamma + 1}{2}}}{\left(\frac{1}{\nu^2} + \frac{z^2 / \nu^2}{\gamma}\right)^{-\frac{\gamma + 1}{2}}} \leq 1.
 \end{align*}
\end{proof}

\begin{proof}[Proof of \autoref{thm:1}] We start with the proof of Result (a), which is quite lengthy. We next turn to the proofs of Results (b) and (c) which are shorter. Recall that we assume that $|\mathrm{O}^\mathsf{c}| \geq \max\{n/2+(p-1/2), |\mathrm{O}|\gamma + p + 2\}$, which implies that $|\mathrm{O}^\mathsf{c}| \geq n/2+(p-1/2)$ which is equivalent to $|\mathrm{O}^\mathsf{c}| \geq |\mathrm{O}|+2p-1$ because $n = |\mathrm{O}^\mathsf{c}| + |\mathrm{O}|$. For the proof, we will assume that $|\mathrm{O}| \geq 1$, i.e.\ that there is at least one outlier, otherwise the proof is trivial.

First, we note that $m(\mathbf{y}) < \infty$ for all $\omega$ and $m(\mathbf{y}_{\mathrm{O}^\mathsf{c}}) < \infty$ because $|\mathrm{O}^\mathsf{c}| \geq |\mathrm{O}|\gamma + p + 2$, which implies that $n > p + 1$. Therefore, both $\pi(\, \cdot \,, \, \cdot \mid \mathbf{y})$ (for all $\omega$) and $\pi(\, \cdot \,, \, \cdot \mid \mathbf{y}_{\mathrm{O}^\mathsf{c}})$ are proper.

We next observe that
\begin{align}
   \frac{m(\mathbf{y})}{m(\mathbf{y}_{\mathrm{O}^\mathsf{c}})\prod_{i \in \mathrm{O}} f(y_i)} &= \frac{m(\mathbf{y})}{m(\mathbf{y}_{\mathrm{O}^\mathsf{c}})\prod_{i \in \mathrm{O}} f(y_i)}
      \int_{\re^p}\int_{0}^{\infty}\pi(\boldsymbol\beta, \sigma \mid \mathbf{y}) \, \d\sigma \, \d\boldsymbol\beta \nonumber \\
   &=   \int_{\re^p}\int_{0}^{\infty}\frac{\pi(\boldsymbol\beta, \sigma) \prod_{i=1}^{n}
        (1/\sigma)f((y_i-\mathbf{x}_i^T\boldsymbol\beta)/\sigma)}{m(\mathbf{y}_{\mathrm{O}^\mathsf{c}}) \prod_{i \in \mathrm{O}} f(y_i)} \, \d\sigma \, \d\boldsymbol\beta \nonumber \\
   &=   \int_{\re^p}\int_{0}^{\infty}
        \pi(\boldsymbol\beta, \sigma\mid \mathbf{y}_{\mathrm{O}^\mathsf{c}})  \prod_{i \in \mathrm{O} } \frac{(1/\sigma)f((y_i-\mathbf{x}_i^T\boldsymbol\beta)/\sigma)}{\sigma^\gamma f(y_i)} \, \d\sigma \, \d\boldsymbol\beta. \label{eq:fct_thm1}
\end{align}
We show that the last integral converges towards 1 as $\omega\rightarrow\infty$ to prove Result~(a).
Let us assume for now that we are allowed to interchange the limit $\omega\rightarrow\infty$ and the integral:
\begin{align*}
&\lim_{\omega\rightarrow\infty}\int_{\re^p}\int_{0}^{\infty}
        \pi(\boldsymbol\beta, \sigma\mid \mathbf{y}_{\mathrm{O}^\mathsf{c}})  \prod_{i \in \mathrm{O} } \frac{(1/\sigma)f((y_i-\mathbf{x}_i^T\boldsymbol\beta)/\sigma)}{\sigma^\gamma f(y_i)} \, \d\sigma \, \d\boldsymbol\beta \cr
 &\qquad = \int_{\re^p}\int_{0}^{\infty}\lim_{\omega\rightarrow\infty}
  \pi(\boldsymbol\beta, \sigma\mid \mathbf{y}_{\mathrm{O}^\mathsf{c}})  \prod_{i \in \mathrm{O} } \frac{(1/\sigma)f((y_i-\mathbf{x}_i^T\boldsymbol\beta)/\sigma)}{\sigma^\gamma f(y_i)} \, \d\sigma \, \d\boldsymbol\beta  \\
 &\qquad = \int_{\re^p}\int_{0}^{\infty}
  \pi(\boldsymbol\beta, \sigma\mid \mathbf{y}_{\mathrm{O}^\mathsf{c}}) \, \d\sigma \, \d\boldsymbol\beta = 1,
\end{align*}
using \eqref{eqn:limit_PDF} and \autoref{prop:proper_limiting}. To prove that we are indeed allowed to interchange the limit and the integral, we use Lebesgue's dominated convergence theorem. Note that pointwise convergence is sufficient, for any value of $\boldsymbol\beta\in\re^p$ and $\sigma>0$, once the limit is inside the integral. In order to use Lebesgue's dominated convergence theorem, we need to  prove that the integrand is bounded above by an integrable function of $\boldsymbol\beta$ and $\sigma$ that does not depend on $\omega$, for any value of $\omega\ge \yo$, where $\yo$ is a constant. The constant $\yo$ can be chosen as large as we want, and minimum values for $\yo$ will be given throughout the proof.

 As mentioned, the main difficulty in proving \autoref{thm:1} is to find an upper bound for the integrand. To facilitate the understanding of the technical arguments that follow, we now present a sketch of how we will proceed. The goal is essentially to get rid of the terms $f(y_i)$ for all $i \in \mathrm{O}$ in the denominator of the integrand in \eqref{eq:fct_thm1} because those are small (given that such $y_i$ behave as $\omega$ which we view as being large). Consider $i \in \mathrm{O}$; when $|\mathbf{x}_i^T\boldsymbol\beta| \leq \omega / 2$, say, $|y_i-\mathbf{x}_i^T\boldsymbol\beta|$ is large and behaves like $\omega /2$, and for $\sigma \geq 1$, \autoref{lemma:bound_student} can be used to bound
\[
\frac{(1/\sigma)f((y_i-\mathbf{x}_i^T\boldsymbol\beta)/\sigma)}{\sigma^\gamma f(y_i)}.
\]
 When $|\mathbf{x}_i^T\boldsymbol\beta| > \omega / 2$, we are not guaranteed that $|y_i-\mathbf{x}_i^T\boldsymbol\beta|$ is large and cannot thus use the PDF term of the outlier to bound $1/f(y_i)$. We thus have to resort to non-outliers. With non-outliers, we consider $y_j$ as fixed, and it is only when $|\mathbf{x}_j^T\boldsymbol\beta|$ is large that the PDF terms $(1/\sigma)f((y_j-\mathbf{x}_j^T\boldsymbol\beta)/\sigma)$ can be used to bound $1/f(y_i)$.
 The strategy used below to deal with the situation is to divide the parameter space  in mutually exclusive areas for which we know exactly in which case we are: either we can use the outlier PDF term to bound $1/f(y_i)$ or not; in the latter case, we know that we have sufficiently non-outliers such that $|\mathbf{x}_j^T\boldsymbol\beta|$ is large. To have a precise control over the number of non-outliers such that $|\mathbf{x}_j^T\boldsymbol\beta|$ is large, we prove that when $|\mathbf{x}_i^T\boldsymbol\beta| > \omega / 2$, implying that we cannot use the PDF term of the outlier to bound $1/f(y_i)$, we have a maximum of $p - 1$ non-outliers such that $|\mathbf{x}_j^T\boldsymbol\beta|$ is not large, meaning essentially that when the outlying point $(y_i, \mathbf{x}_i)$ is not so far from the hyperplane defined by $\boldsymbol\beta$, there are a maximum of $p-1$ non-outlying points $(y_j, \mathbf{x}_j)$ that are not so far from the hyperplane defined by $\boldsymbol\beta$. To understand why that happens, think of how $p$ points define an hyperplane of dimension $p-1$ (see \autoref{fig:proof}). Using that $|\mathrm{O}^\mathsf{c}| \geq |\mathrm{O}|+2p-1$, we know that at least $|\mathrm{O}|+p$ non-outlying points are such that $|\mathbf{x}_j^T\boldsymbol\beta|$ is large, which is sufficient to bound the terms $1/f(y_i)$ (in the worst case, there are $|\mathrm{O}|$ of these terms), given that $p$ non-outlying points will be used to obtain a finite integral.
 \begin{figure}[ht]
 \centering
   \includegraphics[width=0.45\textwidth]{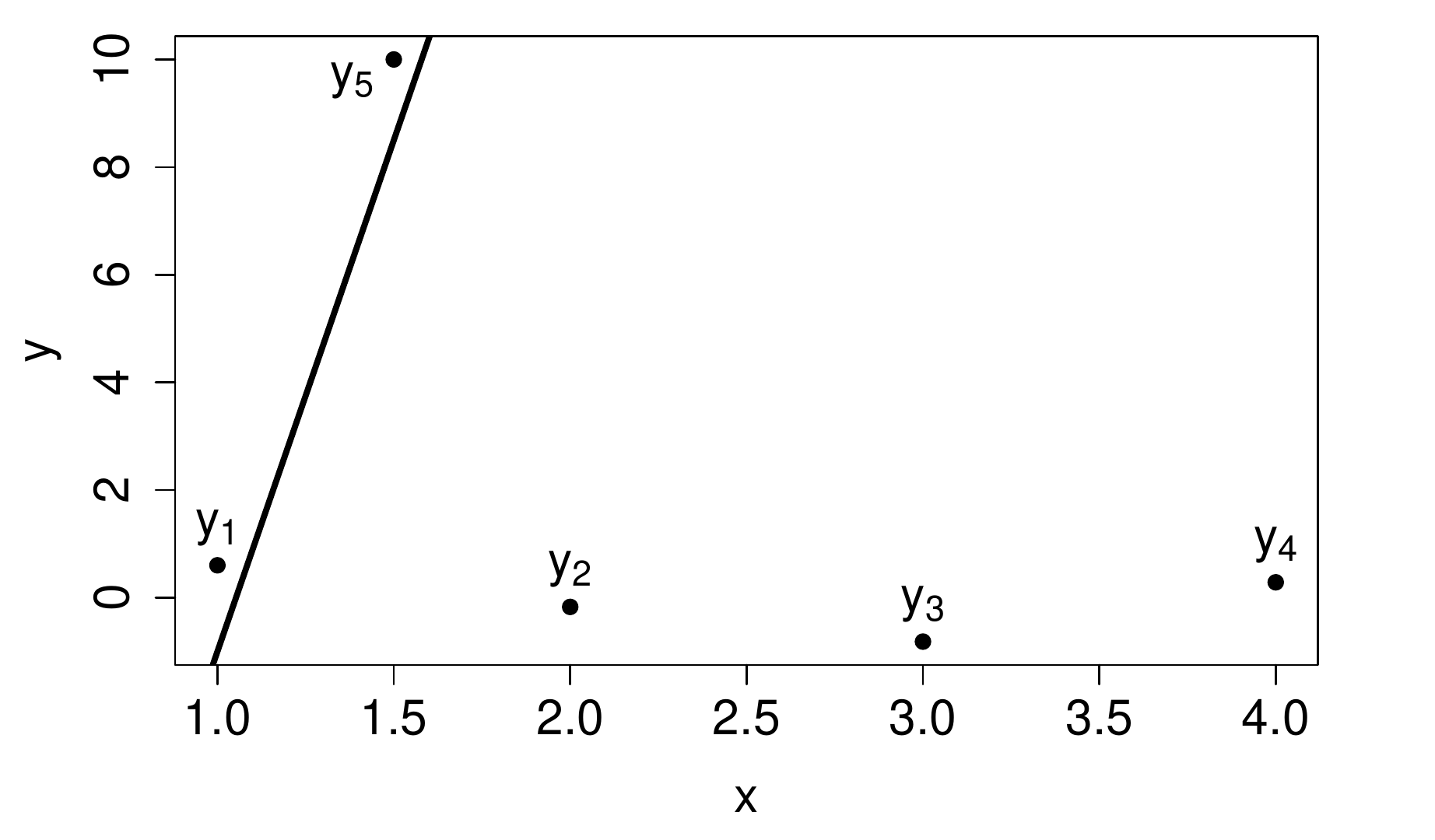}
  \vspace{-3mm}
\caption{Example of a case where a line passes close to a non-outlier and an outlier}\label{fig:proof}
\end{figure}
\normalsize

Let us now continue with the formal proof. In order to bound the integrand in \eqref{eq:fct_thm1}, we first divide the domain of integration of $\sigma$ into two areas: $1\leq\sigma<\infty$ and $0<\sigma<1$. We want to separately analyse the area where the ratio $1/\sigma$ approaches infinity.

We assumed that $y_i$ can be written as $y_i=a_i+b_i \omega$, where $\omega\rightarrow\infty$, and $b_i$ is a constant such that $b_i\ne 0$ if $i \in \mathrm{O}$. Therefore, the ranking of the elements in the set $\{|y_i| : i \in \mathrm{O}\}$ is primarily determined by the values $|b_1|,\ldots,|b_n|$, and we can choose the constant $\yo$ larger than a certain threshold to ensure that this ranking remains unchanged for all $\omega\ge \yo$. Without loss of generality, we assume for convenience that
\begin{equation*}
\omega=\min_{\{i:\,i \in \mathrm{O}\}}|y_i| \hspace{5mm} \text{ and consequently }\hspace{5mm} \min_{\{i:\,i \in \mathrm{O}\}} |b_i|=1.
\end{equation*}

We now bound above the integrand on the first area.

\textbf{Area~1:} Consider $1\le\sigma<\infty$ and assume without loss of generality that $y_1,\ldots,$ $y_{|\mathrm{O}|+2p-1}$ are $|\mathrm{O}|+2p-1$ non-outliers (therefore $\{1, \ldots, |\mathrm{O}|+2p-1\} \subset \mathrm{O}^\mathsf{c}$). We have
\begin{align*}
&\pi(\boldsymbol\beta, \sigma\mid \mathbf{y}_{\mathrm{O}^\mathsf{c}})  \prod_{i \in \mathrm{O} } \frac{(1/\sigma)f((y_i-\mathbf{x}_i^T\boldsymbol\beta)/\sigma)}{\sigma^\gamma f(y_i)} \cr
&\leq C^{|\mathrm{O}^\mathsf{c}| - (|\mathrm{O}|+2p-1)} \frac{\pi(\boldsymbol\beta,\sigma)}{m(\mathbf{y}_{\mathrm{O}^\mathsf{c}}) \, \sigma^{|\mathrm{O}^\mathsf{c}| - p - |\mathrm{O}|\gamma}} \prod_{i=1}^{p} (1/\sigma)f((y_i-\mathbf{x}_i^T\boldsymbol\beta)/\sigma)  \prod_{i = p + 1}^{|\mathrm{O}|+2p-1} f((y_i-\mathbf{x}_i^T\boldsymbol\beta)/\sigma) \prod_{i \in \mathrm{O} } \frac{(1/\sigma)f((y_i-\mathbf{x}_i^T\boldsymbol\beta)/\sigma)}{\sigma^\gamma f(y_i)} \cr
&\leq C^{|\mathrm{O}^\mathsf{c}| - |\mathrm{O}|-2p+2} \frac{1}{m(\mathbf{y}_{\mathrm{O}^\mathsf{c}}) \, \sigma^{|\mathrm{O}^\mathsf{c}| - p - |\mathrm{O}|\gamma}} \prod_{i=1}^{p} (1/\sigma)f((y_i-\mathbf{x}_i^T\boldsymbol\beta)/\sigma)  \prod_{i = p + 1}^{|\mathrm{O}|+2p-1} f((y_i-\mathbf{x}_i^T\boldsymbol\beta)/\sigma) \prod_{i \in \mathrm{O} } \frac{(1/\sigma)f((y_i-\mathbf{x}_i^T\boldsymbol\beta)/\sigma)}{\sigma^\gamma f(y_i)},
\end{align*}
using that we can choose $C$ such that $f \leq C$ and that $\pi(\boldsymbol\beta, \sigma) \leq C\max(1, 1 / \sigma) = C$.

Now, we prove that
\[
 \frac{1}{\sigma^{|\mathrm{O}^\mathsf{c}| - p - |\mathrm{O}|\gamma}} \prod_{i=1}^{p} (1/\sigma)f((y_i-\mathbf{x}_i^T\boldsymbol\beta)/\sigma)
\]
is an integrable function of $\boldsymbol\beta$ and $\sigma$ (note that it does not depend on $\omega$), and next we prove that
\begin{align}\label{eqn:fct_to_bound}
 \prod_{i = p + 1}^{|\mathrm{O}|+2p-1} f((y_i-\mathbf{x}_i^T\boldsymbol\beta)/\sigma) \prod_{i \in \mathrm{O} } \frac{(1/\sigma)f((y_i-\mathbf{x}_i^T\boldsymbol\beta)/\sigma)}{\sigma^\gamma f(y_i)}
\end{align}
is bounded above by a constant. We have
\begin{align*}
&\int_{1}^{\infty}\frac{1}{\sigma^{|\mathrm{O}^\mathsf{c}| - p - |\mathrm{O}|\gamma}} \int_{\re^p}\prod_{i=1}^{p} (1/\sigma)f((y_i-\mathbf{x}_i^T\boldsymbol\beta)/\sigma) \, \d\boldsymbol\beta \, \d\sigma \cr
&\qquad=\abs{\text{det}\left(\begin{array}{c}\mathbf{x}_1^T \cr \vdots \cr \mathbf{x}_p^T\end{array}\right)}^{-1}  \int_{1}^{\infty}\frac{1}{\sigma^{|\mathrm{O}^\mathsf{c}| - p - |\mathrm{O}|\gamma}}\, \d\sigma<\infty,
\end{align*}
using the change of variables $u_i = (y_i - \mathbf{x}_i^T\boldsymbol\beta) / \sigma$, $i=1,\ldots,p$, and that $|\mathrm{O}^\mathsf{c}| - p - |\mathrm{O}|\gamma\geq  2$. The determinant is different from 0 because all the explanatory variables are continuous.

In order to prove that \eqref{eqn:fct_to_bound} is bounded by a constant, we split the domain of $\boldsymbol\beta$ as follows:
\begin{align}\label{eqn_domain_beta}
 &\re^p=\left[\cap_i \mathcal{O}_i^\mathsf{c} \right]\cup \left[\cup_i \left( \mathcal{O}_i\cap\left(\cap_{i_1} \mathcal{F}_{i_1}^\mathsf{c}\right)\right)\right]\cup \left[\cup_{i,i_1}\left(\mathcal{O}_i\cap \mathcal{F}_{i_1}\cap\left(\cap_{i_2\neq i_1}\mathcal{F}_{i_2}^\mathsf{c}\right)\right)\right] \cr
 &\cup\cdots\cup \left[\cup_{i,i_1,\ldots,i_{p-1} (i_j\neq i_s \, \forall i_j,i_s \text{ s.t. } j\neq s)}\left(\mathcal{O}_i\cap \mathcal{F}_{i_1}\cap \cdots\cap\mathcal{F}_{i_{p-1}}\cap \left(\cap_{i_p\neq i_1,\ldots,i_{p-1}}\mathcal{F}_{i_p}^\mathsf{c}\right)\right)\right] \cr
 &\qquad\qquad\qquad \cup \left[\cup_{i,i_1,\ldots,i_{p} (i_j\neq i_s \, \forall i_j,i_s \text{ s.t. } j\neq s)}\left(\mathcal{O}_i\cap \mathcal{F}_{i_1}\cap \cdots\cap\mathcal{F}_{i_{p}}\right)\right],
\end{align}
where
\begin{align}\label{def_O_i}
\mathcal{O}_i&:=\{\boldsymbol\beta:|y_i - \mathbf{x}_i^T\boldsymbol\beta| < \omega/2\},\forall i\in \mathcal{I}_\mathcal{O},
\end{align}
\begin{align}\label{def_F_i}
\mathcal{F}_i&:=\{\boldsymbol\beta:|\mathbf{x}_i^T\boldsymbol\beta|<\omega/\kappa\},\forall i\in\mathcal{I}_\mathcal{F},
\end{align}
 $\mathcal{I}_\mathcal{O}:= \mathrm{O}$ and $\mathcal{I}_\mathcal{F}:=\{p+1,\ldots,|\mathrm{O}|+2p-1\}$ being the sets of indexes of outliers and remaining fixed observations (non-outliers) among observations $1$ to $|\mathrm{O}|+2p-1$, respectively, $\kappa$ being a positive constant to be defined. We find an upper bound on each of these subsets, and because there is a finite number of subsets, we will be able to bound \eqref{eqn:fct_to_bound} by the maximal bound. Note that we change the notation of $\mathrm{O}$ to $\mathcal{I}_\mathcal{O}$ to be aligned with $\mathcal{I}_\mathcal{F}$; it will facilitate the reading of the rest of the proof.

 The set $\mathcal{O}_{i}$ represents the hyperplanes $\mathbf{x}_i^T\boldsymbol\beta$ characterized by the different values of $\boldsymbol\beta$ that satisfy $|y_i - \mathbf{x}_i^T\boldsymbol\beta|< \omega/2$. In other words, it represents the hyperplanes that pass at a vertical distance of less than $\omega/2$ of the point $(\mathbf{x}_i, y_i)$, which is considered as an outlier given that $i\in \mathcal{I}_\mathcal{O}$. Analogously, the set $\mathcal{F}_{i}$ represents the hyperplanes that pass at a vertical distance of less than $\omega/\kappa$ of the point $(\mathbf{x}_i, 0)$, which is considered to be a non-outlier. Therefore, the set $\cap_i \mathcal{O}_i^\mathsf{c}$ represents the hyperplanes that pass at a vertical distance of at least $\omega/2$ of all the points $(\mathbf{x}_i, y_i)$ with $i\in \mathcal{I}_\mathcal{O}$ (all the outliers). The set $\cup_i (\mathcal{O}_i\cap(\cap_{i_1} \mathcal{F}_{i_1}^\mathsf{c}))$ represents the hyperplanes that pass at a vertical distance of less than $\omega/2$ of at least one outlier $(\mathbf{x}_i, y_i)$, but at a vertical distance of at least $\omega/\kappa$ of all the points $(\mathbf{x}_i, 0)$ (the non-outliers). For each $i_1\in\mathcal{I}_\mathcal{F}$, the set $\cup_i(\mathcal{O}_i\cap\mathcal{F}_{i_1}\cap(\cap_{i_2\neq i_1} \mathcal{F}_{i_2}^\mathsf{c}))$ represents the hyperplanes that pass at a vertical distance of less than $\omega/2$ of at least one outlier $(\mathbf{x}_i, y_i)$, at a vertical distance of less than $\omega/\kappa$ of the point $(\mathbf{x}_{i_1}, 0)$ (a non-outlier), but at a vertical distance of at least $\omega/\kappa$ of all the other non-outliers. And so on.

 Now, we claim that $\mathcal{O}_i\cap \mathcal{F}_{i_1}\cap \cdots\cap\mathcal{F}_{i_{p}}=\varnothing$ for all $i,i_1,\ldots,i_{p}$ with $i_j\neq i_s,\forall i_j,i_s$ such that $j\neq s$, meaning that there is no hyperplane that passes at a vertical distance of less than $\omega/2$ of the outlier $(\mathbf{x}_i, y_i)$ and at the same time at a vertical distance of less than $\omega/\kappa$ of $p$ points $(\mathbf{x}_{i_j},0)$. To prove this, we use the fact that $\mathbf{x}_i$ (a vector of size $p$) can be expressed as a linear combination of $\mathbf{x}_{i_1},\ldots,\mathbf{x}_{i_{p}}$. This is true because all explanatory variables are continuous, therefore the space spanned by the vectors $\mathbf{x}_{i_1}, \ldots, \mathbf{x}_{i_p}$ has dimension $p$. As a result, considering that $\boldsymbol\beta\in \mathcal{F}_{i_1}\cap \cdots\cap\mathcal{F}_{i_{p}}$ and $\mathbf{x}_i=\sum_{s=1}^{p} c_s \mathbf{x}_{i_s}$ for some $c_1, \ldots, c_{p}\in\re$, we have
\begin{align*}
|y_i - \mathbf{x}_i^T\boldsymbol\beta| =  \left|y_i - \left(\sum_{s=1}^{p} c_s \mathbf{x}_{i_s}\right)^T \boldsymbol\beta\right|\za{\geq} \left| |y_i| - \left|\sum_{s=1}^{p}c_s\mathbf{x}_{i_s}^T\boldsymbol\beta\right| \right| &\zb{\geq}\left|\omega-\frac{\omega}{\kappa}\sum_{s=1}^{p}|c_s|\right| \cr
&\zc{\geq}\omega-\frac{\omega}{2}.
\end{align*}
 In Step $a$, we use the reverse triangle inequality. In Step $b$, we use that $\omega = \min_{\{i:\,i \in \mathrm{O}\}}|y_i|$ and $|\sum_{s=1}^{p} c_s \mathbf{x}_{i_s}^T\boldsymbol\beta|\leq\sum_{s=1}^{p}|c_s||\mathbf{x}_{i_s}^T\boldsymbol\beta| \leq \sum_{s=1}^{p}|c_s|\omega/\kappa$ because $\boldsymbol\beta\in \mathcal{F}_{i_1}\cap \cdots\cap\mathcal{F}_{i_{p}}$, which implies that $|\mathbf{x}_i^T\boldsymbol\beta|<\omega/\kappa$ for all $s \in \{i_1,\ldots,i_p\}$. In Step $c$, we consider that the constant $\kappa$ is such that $\kappa \geq 2 \sum_{s=1}^{p}|c_s|$ (we define $\kappa$ such that it satisfies this inequality for any combination of $i$ and $i_1,\ldots,i_{p}$; without loss of generality we consider that $\kappa \geq 1$). Therefore, we have that if $\boldsymbol\beta\in \mathcal{F}_{i_1}\cap \cdots\cap\mathcal{F}_{i_{p}}$, then $\boldsymbol\beta\notin \mathcal{O}_i$. This proves that $\mathcal{O}_i\cap \mathcal{F}_{i_1}\cap \cdots\cap\mathcal{F}_{i_{p}} = \varnothing$ for all $i,i_1,\ldots,i_{p}$ with $i_j\neq i_s,\forall i_j,i_s$ such that $j\neq s$. This in turn implies that (\ref{eqn_domain_beta}) can be rewritten as
 \begin{align*}
 &\re^p=\left[\cap_i \mathcal{O}_i^\mathsf{c} \right]\cup \left[\cup_i \left( \mathcal{O}_i\cap\left(\cap_{i_1} \mathcal{F}_{i_1}^\mathsf{c}\right)\right)\right]\cup \left[\cup_{i,i_1}\left(\mathcal{O}_i\cap \mathcal{F}_{i_1}\cap\left(\cap_{i_2\neq i_1}\mathcal{F}_{i_2}^\mathsf{c}\right)\right)\right] \cr
 &\cup\cdots\cup \left[\cup_{i,i_1,\ldots,i_{p-1} (i_j\neq i_s \, \forall i_j,i_s \text{ s.t. } j\neq s)}\left(\mathcal{O}_i\cap \mathcal{F}_{i_1}\cap \cdots\cap\mathcal{F}_{i_{p-1}}\cap \left(\cap_{i_p\neq i_1,\ldots,i_{p-1}}\mathcal{F}_{i_p}^\mathsf{c}\right)\right)\right].
\end{align*}
This decomposition of $\re^p$ is comprised of $1+\sum_{i=0}^{p-1} {{|\mathrm{O}| + p - 1} \choose {i}}$ mutually exclusive sets given by $\cap_i \mathcal{O}_i^\mathsf{c}$, $\cup_i ( \mathcal{O}_i\cap(\cap_{i_1} \mathcal{F}_{i_1}^\mathsf{c}))$, $\cup_{i}(\mathcal{O}_i\cap \mathcal{F}_{i_1}\cap(\cap_{i_2\neq i_1}\mathcal{F}_{i_2}^\mathsf{c}))$ for $i_1\in\mathcal{I}_\mathcal{F} $, and so on.

We are now ready to bound \eqref{eqn:fct_to_bound}. We first show that the function is bounded on $\boldsymbol\beta\in \cap_i \mathcal{O}_i^\mathsf{c}$. Recall that $1\le\sigma < \infty$. For all $i\in\mathcal{I}_\mathcal{O}$, we have
$$
 \frac{(1/\sigma) f((y_i - \mathbf{x}_i^T\boldsymbol\beta )/\sigma)}{\sigma^{\gamma} f(y_i)}\leq \frac{f(\omega/(2\sigma))}{\sigma^{\gamma + 1} f(2 |b_i| \omega)}\leq (4|b_i|)^{\gamma + 1},
$$
using the monotonicity of $f$ twice: 1) $|y_i - \mathbf{x}_i^T\boldsymbol\beta|/\sigma\geq \omega/(2\sigma)$, and 2) $|y_i| \leq |a_i + b_i \omega| \leq |a_i| + |b_i| \omega \leq 2 |b_i| \omega$, and then \autoref{lemma:bound_student} with $z = 2|b_i|\omega$ and $\nu = 4 |b_i| \sigma \geq 4$. Therefore, on $\boldsymbol\beta\in \cap_i \mathcal{O}_i^\mathsf{c}$ and $1\le\sigma < \infty$,
$$
\prod_{i = p + 1}^{|\mathrm{O}|+2p-1} f((y_i-\mathbf{x}_i^T\boldsymbol\beta)/\sigma) \prod_{i \in \mathcal{I}_\mathcal{O}} \frac{(1/\sigma)f((y_i-\mathbf{x}_i^T\boldsymbol\beta)/\sigma)}{\sigma^\gamma f(y_i)}\leq C^{|\mathrm{O}|+p-1} \prod_{i \in \mathcal{I}_\mathcal{O}} (4|b_i|)^{\gamma + 1},
$$
using that $C$ can be chosen such that $f \leq C$.

Now, we consider the area defined by: $1\le\sigma < \infty$ and $\boldsymbol\beta$ belongs to one of the $\sum_{i=0}^{p-1} {{|\mathrm{O}| + p - 1} \choose {i}}$ mutually exclusive sets $\cup_i ( \mathcal{O}_i\cap(\cap_{i_1} \mathcal{F}_{i_1}^\mathsf{c}))$, $\cup_{i}(\mathcal{O}_i\cap \mathcal{F}_{i_1}\cap(\cap_{i_2\neq i_1}\mathcal{F}_{i_2}^\mathsf{c}))$ for $i_1\in\mathcal{I}_\mathcal{F} $, etc. We have
\begin{align*}
\prod_{i = p + 1}^{|\mathrm{O}|+2p-1} f((y_i-\mathbf{x}_i^T\boldsymbol\beta)/\sigma) \prod_{i \in \mathcal{I}_\mathcal{O}} \frac{(1/\sigma)f((y_i-\mathbf{x}_i^T\boldsymbol\beta)/\sigma)}{\sigma^\gamma f(y_i)}  &\za{\leq} C^{|\mathrm{O}|+p-1} \frac{\prod_{i = p + 1}^{|\mathrm{O}|+p} (1/\sigma)f((y_i-\mathbf{x}_i^T\boldsymbol\beta)/\sigma)}{\prod_{i \in \mathcal{I}_\mathcal{O}}\sigma^\gamma f(y_i)} \cr
 &\zb{\leq} C^{|\mathrm{O}|+p-1}\prod_{i \in \mathcal{I}_\mathcal{O}} \frac{f(\omega / (2\kappa \sigma))}{\sigma^{\gamma + 1} f(2 |b_i| \omega)} \cr
 &\zc{\leq} C^{|\mathrm{O}|+p-1}\prod_{i \in \mathcal{I}_\mathcal{O}} (4 |b_i| \kappa)^{\gamma + 1}.
\end{align*}
In Step $a$, we use $f \leq C$ for all $i\in \mathcal{I}_\mathcal{O}$. We also use the fact that in any of the sets in which $\boldsymbol\beta$ can belong, there are at least $|\mathrm{O}|$ non-outlying points $(\mathbf{x}_i, y_i)$ such that $|\mathbf{x}_i^T\boldsymbol\beta| \geq \omega / \kappa$. Indeed, the case in which there are the least non-outliers such that $|\mathbf{x}_i^T\boldsymbol\beta| \geq \omega/\kappa$ corresponds to $\boldsymbol\beta\in\cup_i(\mathcal{O}_i\cap \mathcal{F}_{i_1}\cap \cdots\cap\mathcal{F}_{i_{p-1}}\cap (\cap_{i_p\neq i_1,\ldots,i_{p-1}}\mathcal{F}_{i_p}^\mathsf{c}))$. In this case there are $p - 1$ non-outliers such that $|\mathbf{x}_i^T\boldsymbol\beta|<\omega/\kappa$ (observations $i_1$ to $i_{p-1}$), which leaves $|\mathrm{O}|+p-1-(p-1)=|\mathrm{O}|$ non-outliers such that $|\mathbf{x}_i^T\boldsymbol\beta|\geq\omega/\kappa$ (i.e.\ that $\cap_{i_p\neq i_1,\ldots,i_{p-1}}\mathcal{F}_{i_p}^\mathsf{c}$ is an intersection of $|\mathrm{O}|$ sets). We consider without loss of generality that the non-outliers with $|\mathbf{x}_i^T\boldsymbol\beta| \geq \omega / \kappa$ are the observations with indices $p+1, \ldots, |\mathrm{O}|+p$, and we use $f \leq C$ for all $i\in \{|\mathrm{O}|+p + 1, \ldots, |\mathrm{O}|+2p - 1\}$ (if this set is non-empty). In Step $b$, we use the monotonicity of $f$ twice: 1) $|y_i - \mathbf{x}_i^T\boldsymbol\beta| \geq \left||\mathbf{x}_i^T\boldsymbol\beta| - |y_i|\right| \geq \omega / (2\kappa)$ given that for non-outliers $|y_i| = |a_i| \leq \omega / (2\kappa)$, and 2) for the outliers, $|y_i| \leq |a_i + b_i \omega| \leq |a_i| + |b_i| \omega \leq 2 |b_i| \omega$. In Step $c$, we use \autoref{lemma:bound_student} with $z = 2|b_i|\omega$ and $\nu = 4 |b_i| \kappa \sigma \geq 4 $. Note that the argument used here justifies the need of the assumption $|\mathrm{O}^\mathsf{c}| \geq |\mathrm{O}|+2p-1$.

\textbf{Area~2:} Consider $0<\sigma< 1$. We actually need to show that
\begin{align*}
  \lim_{\omega\rightarrow\infty} \int_{\re^p}\int_{0}^{1}
  \pi(\boldsymbol\beta, \sigma\mid \mathbf{y}_{\mathrm{O}^\mathsf{c}})  \prod_{i \in \mathrm{O} } \frac{(1/\sigma)f((y_i-\mathbf{x}_i^T\boldsymbol\beta)/\sigma)}{\sigma^\gamma f(y_i)} \, \d\sigma \, \d\boldsymbol\beta = \int_{\re^p}\int_{0}^{1}
  \pi(\boldsymbol\beta, \sigma\mid \mathbf{y}_{\mathrm{O}^\mathsf{c}}) \, \d\sigma \, \d\boldsymbol\beta.
\end{align*}
For Area~2, we proceed in a slightly different manner than for Area~1. We begin by dividing the first integral above into two parts as follows:
\begin{align*}
  &\lim_{\omega\rightarrow\infty}\int_{\re^p}\int_{0}^{1}
  \pi(\boldsymbol\beta, \sigma\mid \mathbf{y}_{\mathrm{O}^\mathsf{c}})  \prod_{i \in \mathrm{O} } \frac{(1/\sigma)f((y_i-\mathbf{x}_i^T\boldsymbol\beta)/\sigma)}{\sigma^\gamma f(y_i)} \, \d\sigma \, \d\boldsymbol\beta \\
  &\quad=\lim_{\omega\rightarrow\infty} \int_{\re^p}\int_{0}^{1}
  \pi(\boldsymbol\beta, \sigma\mid \mathbf{y}_{\mathrm{O}^\mathsf{c}})  \prod_{i \in \mathrm{O} } \frac{(1/\sigma)f((y_i-\mathbf{x}_i^T\boldsymbol\beta)/\sigma)}{\sigma^\gamma f(y_i)} \,
  \ind_{\cap_i \mathcal{O}_i^\mathsf{c}}(\boldsymbol\beta) \, \d\sigma \, \d\boldsymbol\beta \\
  &\qquad+\lim_{\omega\rightarrow\infty} \int_{\cup_i \mathcal{O}_i}\int_{0}^{1}
  \pi(\boldsymbol\beta, \sigma\mid \mathbf{y}_{\mathrm{O}^\mathsf{c}})  \prod_{i \in \mathrm{O} } \frac{(1/\sigma)f((y_i-\mathbf{x}_i^T\boldsymbol\beta)/\sigma)}{\sigma^\gamma f(y_i)} \, \d\sigma \, \d\boldsymbol\beta,
\end{align*}
where $\ind_{\cap_i \mathcal{O}_i^\mathsf{c}}(\boldsymbol\beta)$ is the indicator function that, in this case, takes the value 1 if $\boldsymbol\beta \in \cap_i \mathcal{O}_i^\mathsf{c}$, and $0$ otherwise, $\mathcal{O}_i$ being defined as before. We show that the limit of the first integral on the right-hand side (RHS) is equal to $\int_{\re^p}\int_{0}^{1}
  \pi(\boldsymbol\beta, \sigma\mid \mathbf{y}_{\mathrm{O}^\mathsf{c}}) \, \d\sigma \, \d\boldsymbol\beta$ and that the limit of the second integral is equal to 0.

  For the first part, we again use Lebesgue's dominated convergence theorem in order to interchange the limit $\omega\rightarrow\infty$ and the integral; assuming the theorem holds,
\begin{align*}
&\lim_{\omega\rightarrow\infty}\int_{\re^p}\int_{0}^{1}
  \pi(\boldsymbol\beta, \sigma\mid \mathbf{y}_{\mathrm{O}^\mathsf{c}})  \prod_{i \in \mathrm{O} } \frac{(1/\sigma)f((y_i-\mathbf{x}_i^T\boldsymbol\beta)/\sigma)}{\sigma^\gamma f(y_i)} \,
  \ind_{\cap_i \mathcal{O}_i^\mathsf{c}}(\boldsymbol\beta) \, \d\sigma \, \d\boldsymbol\beta\cr
 &\quad = \int_{\re^p}\int_{0}^{1}
  \pi(\boldsymbol\beta, \sigma\mid \mathbf{y}_{\mathrm{O}^\mathsf{c}}) \lim_{\omega \rightarrow \infty} \prod_{i \in \mathrm{O} } \frac{(1/\sigma)f((y_i-\mathbf{x}_i^T\boldsymbol\beta)/\sigma)}{\sigma^\gamma f(y_i)} \,
  \ind_{\cap_i \mathcal{O}_i^\mathsf{c}}(\boldsymbol\beta) \, \d\sigma \, \d\boldsymbol\beta\cr
 &\quad =\int_{\re^p}\int_{0}^{1}  \pi(\boldsymbol\beta, \sigma\mid \mathbf{y}_{\mathrm{O}^\mathsf{c}}) \times 1 \times \ind_{\re^p}(\boldsymbol\beta)
  \, \d\sigma \, \d\boldsymbol\beta = \int_{\re^p}\int_{0}^{1}
  \pi(\boldsymbol\beta, \sigma\mid \mathbf{y}_{\mathrm{O}^\mathsf{c}}) \, \d\sigma \, \d\boldsymbol\beta,
\end{align*}
using \eqref{eqn:limit_PDF}, and $\lim_{\omega\rightarrow\infty}$ $ \ind_{\cap_i\mathcal{O}_i^\mathsf{c}}(\boldsymbol\beta)=\ind_{\re^p}(\boldsymbol\beta)=1\Leftrightarrow \lim_{\omega\rightarrow\infty} \ind_{\cup_i\mathcal{O}_i}(\boldsymbol\beta)=0$. Indeed, if $i\in \mathcal{I}_\mathcal{O}$ and $b_i>0$ (which implies that $y_i>0$), $\boldsymbol\beta\in\mathcal{O}_i$ implies that $|y_i-\mathbf{x}_i^T\boldsymbol\beta|< \omega/2\leq y_i/2$, which in turn implies that  $y_i/2<\mathbf{x}_i^T\boldsymbol\beta<3y_i/2$, and in the limit (implying $y_i \rightarrow \infty$ in this case), no fixed $\boldsymbol\beta\in\re^p$ satisfies the inequality; we have the same conclusion if $b_i<0$. Note that pointwise convergence is sufficient, for any value of $\boldsymbol\beta\in\re^p$ and $1 >\sigma>0$, once the limit is inside the integral. We now demonstrate that the integrand is bounded above, for any value of $\omega\ge \yo$, by an integrable function of $\boldsymbol\beta$ and $\sigma$ that does not depend on $\omega$.

Consider that $\boldsymbol\beta\in \cap_i \mathcal{O}_i^\mathsf{c}$, that is $\{\boldsymbol\beta:|y_i-\mathbf{x}_i^T\boldsymbol\beta|\geq \omega/2 \text{ for all } i\in\mathcal{I}_{\mathcal{O}}\}$, and $0<\sigma< 1$. Note that the integrand is equal to 0 if $\boldsymbol\beta\notin \cap_i \mathcal{O}_i^\mathsf{c}$. We have
\begin{align*}
 \pi(\boldsymbol\beta, \sigma\mid \mathbf{y}_{\mathrm{O}^\mathsf{c}})  \prod_{i \in \mathrm{O} } \frac{(1/\sigma)f((y_i-\mathbf{x}_i^T\boldsymbol\beta)/\sigma)}{\sigma^\gamma f(y_i)} & \za{\leq}  \pi(\boldsymbol\beta, \sigma) \left[\prod_{i \in \text{O}^\mathsf{c}} (1 / \sigma) f((y_i - \mathbf{x}_i^T \boldsymbol\beta) / \sigma)\right]  \prod_{i \in \mathrm{O} } \frac{f(\omega / 2)}{f(y_i)} \cr
 &\zb{\leq}  \pi(\boldsymbol\beta, \sigma) \left[\prod_{i \in \text{O}^\mathsf{c}} (1 / \sigma) f((y_i - \mathbf{x}_i^T \boldsymbol\beta) / \sigma)\right]  \prod_{i \in \mathrm{O} } \frac{f(\omega / 2)}{f(2 |b_i| \omega)} \cr
 &\zc{\leq}  \pi(\boldsymbol\beta, \sigma) \left[\prod_{i \in \text{O}^\mathsf{c}} (1 / \sigma) f((y_i - \mathbf{x}_i^T \boldsymbol\beta) / \sigma)\right]  \prod_{i \in \mathrm{O} } (4 |b_i|)^{\gamma + 1},
\end{align*}
which is an integrable function. Indeed, it is proportional to the numerator of a posterior density with a prior given by $\pi(\, \cdot \,, \cdot \,)$ and based on $|\text{O}^\mathsf{c}|$ data points, which is integrable because $|\text{O}^\mathsf{c}| \geq |\mathrm{O}|\gamma + p + 2 > p + 1$, by \autoref{prop:proper}. In Step $a$, we use the monotonicity of the tails of the function $z \mapsto |z| f(z)$ and next the monotonicity of $f$, because $|y_i-\mathbf{x}_i^T\boldsymbol\beta| / \sigma\geq |y_i-\mathbf{x}_i^T\boldsymbol\beta| \geq \omega/2\geq\yo/2$, implying $(|y_i-\mathbf{x}_i^T\boldsymbol\beta|/\sigma)f((y_i-\mathbf{x}_i^T\boldsymbol\beta)/\sigma) \leq |y_i-\mathbf{x}_i^T\boldsymbol\beta|f(y_i-\mathbf{x}_i^T\boldsymbol\beta) \Leftrightarrow (1/\sigma)f((y_i-\mathbf{x}_i^T\boldsymbol\beta)/\sigma) \leq f(y_i-\mathbf{x}_i^T\boldsymbol\beta) \leq f(\omega / 2)$. In Step $b$, we use the monotonicity of $f$: $|y_i| \leq |a_i + b_i \omega| \leq |a_i| + |b_i| \omega \leq 2 |b_i| \omega$. In Step $c$, we use \autoref{lemma:bound_student} with $z = 2|b_i|\omega$ and $\nu = 4 |b_i| \geq 4$.

We now prove that
\begin{equation*}
\lim_{\omega\rightarrow\infty} \int_{\cup_i \mathcal{O}_i}\int_{0}^{1}
  \pi(\boldsymbol\beta, \sigma\mid \mathbf{y}_{\mathrm{O}^\mathsf{c}})  \prod_{i \in \mathrm{O} } \frac{(1/\sigma)f((y_i-\mathbf{x}_i^T\boldsymbol\beta)/\sigma)}{\sigma^\gamma f(y_i)} \, \d\sigma \, \d\boldsymbol\beta = 0.
  \end{equation*}
  We first bound above the integrand and then we prove that the integral of the upper bound converges towards 0 as $\omega\rightarrow\infty$. Using a similar strategy as when we bounded \eqref{eqn:fct_to_bound}, we split the domain of $\boldsymbol\beta$ as follows:
\begin{align*}
 &\cup_i \mathcal{O}_i=\left[\cup_i \left( \mathcal{O}_i\cap\left(\cap_{i_1} \mathcal{F}_{i_1}^\mathsf{c}\right)\right)\right]\cup \left[\cup_{i,i_1}\left(\mathcal{O}_i\cap \mathcal{F}_{i_1}\cap\left(\cap_{i_2\neq i_1}\mathcal{F}_{i_2}^\mathsf{c}\right)\right)\right] \cr
 &\cup\cdots\cup \left[\cup_{i,i_1,\ldots,i_{p-1} (i_j\neq i_s \, \forall i_j,i_s \text{ s.t. } j\neq s)}\left(\mathcal{O}_i\cap \mathcal{F}_{i_1}\cap \cdots\cap\mathcal{F}_{i_{p-1}}\cap \left(\cap_{i_p\neq i_1,\ldots,i_{p-1}}\mathcal{F}_{i_p}^\mathsf{c}\right)\right)\right] \cr
 &\qquad\qquad\cup \left[\cup_{i,i_1,\ldots,i_{p} (i_j\neq i_s \, \forall i_j,i_s \text{ s.t. } j\neq s)}\left(\mathcal{O}_i\cap \mathcal{F}_{i_1}\cap \cdots\cap\mathcal{F}_{i_{p}}\right)\right],
\end{align*}
$\mathcal{F}_{i}$ being defined as before, but we consider that $\mathcal{I}_\mathcal{F} = \{1,\ldots, |\text{O}| + 2p - 1\}$ (we assume as previously that the first $|\text{O}| + 2p - 1$ data points are non-outliers; recall that $|\mathrm{O}^\mathsf{c}| \geq |\mathrm{O}|+2p-1$). We can use the same argument as before to show that $\mathcal{O}_i\cap \mathcal{F}_{i_1}\cap \cdots\cap\mathcal{F}_{i_{p}}=\varnothing$ for all $i,i_1,\ldots,i_{p}$ with $i_j\neq i_s$, $\forall i_j\neq i_s$ such that $j\neq s$. Therefore,
 \begin{align*}
 &\cup_i \mathcal{O}_i=\left[\cup_i \left( \mathcal{O}_i\cap\left(\cap_{i_1} \mathcal{F}_{i_1}^\mathsf{c}\right)\right)\right]\cup \left[\cup_{i,i_1}\left(\mathcal{O}_i\cap \mathcal{F}_{i_1}\cap\left(\cap_{i_2\neq i_1}\mathcal{F}_{i_2}^\mathsf{c}\right)\right)\right] \cr
 &\cup\cdots\cup \left[\cup_{i,i_1,\ldots,i_{p-1} (i_j\neq i_s \, \forall i_j,i_s \text{ s.t. } j\neq s)}\left(\mathcal{O}_i\cap \mathcal{F}_{i_1}\cap \cdots\cap\mathcal{F}_{i_{p-1}}\cap \left(\cap_{i_p\neq i_1,\ldots,i_{p-1}}\mathcal{F}_{i_p}^\mathsf{c}\right)\right)\right].
\end{align*}

This decomposition of $\cup_i \mathcal{O}_i$ is comprised of $\sum_{i=0}^{p-1} {{|\text{O}|+2p-1} \choose {i}}$ mutually exclusive sets given by $\cup_i ( \mathcal{O}_i\cap(\cap_{i_1} \mathcal{F}_{i_1}^\mathsf{c}))$, $\cup_{i}(\mathcal{O}_i\cap \mathcal{F}_{i_1}\cap(\cap_{i_2\neq i_1}\mathcal{F}_{i_2}^\mathsf{c}))$ for $i_1\in\mathcal{I}_\mathcal{F} $, and so on. We now consider the area defined by: $0<\sigma< 1$ and  $\boldsymbol\beta$ belongs to one of these $\sum_{i=0}^{p-1} {{|\text{O}|+2p-1} \choose {i}}$ mutually exclusive sets. We have
\begin{align*}
 &\pi(\boldsymbol\beta, \sigma\mid \mathbf{y}_{\mathrm{O}^\mathsf{c}})  \prod_{i \in \mathrm{O} } \frac{(1/\sigma)f((y_i-\mathbf{x}_i^T\boldsymbol\beta)/\sigma)}{\sigma^\gamma f(y_i)} \cr
 &\propto \pi(\boldsymbol\beta, \sigma) \left[\prod_{i \in \text{O}^\mathsf{c}} (1 / \sigma) f((y_i - \mathbf{x}_i^T \boldsymbol\beta) / \sigma)\right] \prod_{i \in \mathrm{O} } \frac{(1/\sigma)f((y_i-\mathbf{x}_i^T\boldsymbol\beta)/\sigma)}{f(y_i)} \cr
 &\za{\leq} (C / \sigma) \left[\prod_{i \in \text{O}^\mathsf{c}} (1 / \sigma) f((y_i - \mathbf{x}_i^T \boldsymbol\beta) / \sigma)\right] \prod_{i \in \mathrm{O} } \frac{(1/\sigma)f((y_i-\mathbf{x}_i^T\boldsymbol\beta)/\sigma)}{f(y_i)} \cr
 &\zb{\leq} (C / \sigma)[(1/\sigma) f(\omega / (2\kappa\sigma))]  \prod_{i = 1 (i \neq i_p,\ldots,i_{|\text{O}|+p})}^n (1 / \sigma) f((y_i - \mathbf{x}_i^T \boldsymbol\beta) / \sigma) \prod_{i \in \mathrm{O} } \frac{(1/\sigma) f(\omega / (2\kappa\sigma))}{f(y_i)} \cr
 &\zc{\leq} (2 \kappa C^2 / \omega) (1 / \sigma)  \prod_{i = 1 (i \neq i_p,\ldots,i_{|\text{O}|+p})}^n (1 / \sigma) f((y_i - \mathbf{x}_i^T \boldsymbol\beta) / \sigma) \prod_{i \in \mathrm{O} } \frac{f(\omega / (2\kappa))}{f(y_i)} \cr
 &\zd{\leq} (2 \kappa C^2 / \omega) (1 / \sigma)  \prod_{i = 1 (i \neq i_p,\ldots,i_{|\text{O}|+p})}^n (1 / \sigma) f((y_i - \mathbf{x}_i^T \boldsymbol\beta) / \sigma) \prod_{i \in \mathrm{O} } \frac{f(\omega / (2\kappa))}{f(2 |b_i| \omega)} \cr
 &\ze{\leq} (2 \kappa C^2 / \omega) (1 / \sigma)  \prod_{i = 1 (i \neq i_p,\ldots,i_{|\text{O}|+p})}^n (1 / \sigma) f((y_i - \mathbf{x}_i^T \boldsymbol\beta) / \sigma) \prod_{i \in \mathrm{O} } (4 |b_i| \kappa)^{\gamma + 1} \cr
 &\propto (1 / \omega) (1 / \sigma)  \prod_{i = 1 (i \neq i_p,\ldots,i_{|\text{O}|+p})}^n (1 / \sigma) f((y_i - \mathbf{x}_i^T \boldsymbol\beta) / \sigma).
\end{align*}
In Step $a$, we use that $\pi(\boldsymbol\beta, \sigma) \le C\max(1,1/\sigma) = C/\sigma$. In Step $b$, we use that in any of the sets in which $\boldsymbol\beta$ can belong, there are at least $|\text{O}|+1$ non-outlying points such that $|\mathbf{x}_i^T\boldsymbol\beta|\geq\omega/\kappa$ (corresponding to $\boldsymbol\beta\in\mathcal{F}_{i}^\mathsf{c}$ for at least $|\text{O}|+1$ non-outlying points). Indeed, as explained before, the case in which there are the least non-outliers such that $|\mathbf{x}_i^T\boldsymbol\beta|\geq\omega/\kappa$ corresponds to $\boldsymbol\beta\in\cup_i(\mathcal{O}_i\cap \mathcal{F}_{i_1}\cap \cdots\cap\mathcal{F}_{i_{p-1}}\cap (\cap_{i_p\neq i_1,\ldots,i_{p-1}}\mathcal{F}_{i_p}^\mathsf{c}))$. In this case there are $p-1$ non-outliers such that $|\mathbf{x}_i^T\boldsymbol\beta|<\omega/\gamma$, which leaves at least $|\text{O}|+2p-1-(p-1)$ non-outliers such that $|\mathbf{x}_i^T\boldsymbol\beta|\geq\omega/\kappa$ (i.e.\ that $\cap_{i_p\neq i_1,\ldots,i_{p-1}}\mathcal{F}_{i_p}^\mathsf{c}$ is an intersection of $|\text{O}|+2p-1-(p-1)$ sets), and we know that $|\text{O}|+2p-1-(p-1)=|\text{O}|+p \geq |\text{O}| + 1$. This implies that there exists a set of $|\text{O}| + 1$ indices, that is considered to be without loss of generality $\{i_p,\ldots,i_{|\text{O}|+p}\}\subset\mathcal{I}_\mathcal{F}$, such that for all $i\in\{i_p,\ldots,i_{|\text{O}|+p}\}$,
\[
 f((y_i - \mathbf{x}_i^T \boldsymbol\beta)/\sigma)  \leq f(\omega / (2\kappa\sigma)),
\]
using the monotonicity of $f$: $|y_i - \mathbf{x}_i^T\boldsymbol\beta| \geq \left||\mathbf{x}_i^T\boldsymbol\beta| - |y_i|\right| \geq \omega / (2\kappa)$ given that for non-outliers $|y_i| = |a_i| \leq \omega / (2\kappa)$. In Step $c$, we use that $(1/\sigma) f(\omega / (2\kappa\sigma)) \leq (2 \kappa / \omega) C$, because the function $z \mapsto z f(z)$ is bounded above (we can thus choose $C$ to be large enough to bound that function), and the monotonicity of the tails of $z \mapsto |z| f(z)$ to obtain $(\omega / (2\kappa\sigma)) f(\omega / (2\kappa\sigma)) \leq (\omega / (2\kappa)) f(\omega / (2\kappa)) \Leftrightarrow (1 / \sigma) f(\omega / (2\kappa\sigma)) \leq f(\omega / (2\kappa))$ for $|\text{O}|$ terms, because $\omega / (2\kappa\sigma) \geq \omega / (2\kappa) \geq \yo / (2\kappa)$ (which is in the tails of the function for $\yo$ large enough).  In Step $d$, we use the monotonicity of $f$: $|y_i| \leq |a_i + b_i \omega| \leq |a_i| + |b_i| \omega \leq 2 |b_i| \omega$. In Step $e$, we use \autoref{lemma:bound_student} with $z = 2|b_i|\omega$ and $\nu = 4 |b_i| \kappa \geq 4$.

We have that
\begin{align*}
 & (1 / \omega) \int_{\cup_i \mathcal{O}_i}\int_{0}^{1} (1 / \sigma)  \prod_{i = 1 (i \neq i_p,\ldots,i_{|\text{O}|+p})}^n (1 / \sigma) f((y_i - \mathbf{x}_i^T \boldsymbol\beta) / \sigma) \, \d\sigma \, \d\boldsymbol\beta \cr
 &\leq (1 / \omega) \int_{\re^p} \int_{0}^{\infty} (1 / \sigma)  \prod_{i = 1 (i \neq i_p,\ldots,i_{|\text{O}|+p})}^n (1 / \sigma) f((y_i - \mathbf{x}_i^T \boldsymbol\beta) / \sigma) \, \d\sigma \, \d\boldsymbol\beta.
\end{align*}
In order to prove that the term on the RHS vanishes as $\omega\rightarrow\infty$, it suffices to prove that the integral is bounded by a constant that does not depend on $\omega$, because $1/\omega\rightarrow 0$. The integral corresponds to a marginal density of $n - (|\text{O}| + 1) = |\text{O}^\mathsf{c}| - 1$ data points, based on a prior distribution such that $\pi(\boldsymbol\beta, \sigma) = 1 / \sigma$. In the proof of Proposition 2.1 in \cite{gagnon2020}, it is shown that such a marginal density is bounded above by a constant that does not depend on $\omega$ if the number of data points is greater than or equal to $p + 1$, provided that the prior density, divided by $1 / \sigma$ is bounded, which is the case. We have that $|\text{O}^\mathsf{c}| - 1 \geq |\mathrm{O}|\gamma + p +1 \geq p + 1$. Therefore, the marginal density is bounded above by a constant that does not depend on $\omega$. This concludes the proof of Result (a).

We now turn to the proof of Result (b). We have that
\[
 \pi(\boldsymbol\beta, \sigma \mid \mathbf{y}) = \pi(\boldsymbol\beta,\sigma\mid\mathbf{y}_{\mathrm{O}^\mathsf{c}}) \, \frac{m(\mathbf{y}_{\mathrm{O}^\mathsf{c}}) \prod_{i \in \mathrm{O}}f(y_i)}{m(\mathbf{y})} \prod_{i \in \mathrm{O} } \frac{(1/\sigma)f((y_i-\mathbf{x}_i^T\boldsymbol\beta)/\sigma)}{\sigma^\gamma f(y_i)},
\]
and
\[
 \frac{m(\mathbf{y}_{\mathrm{O}^\mathsf{c}}) \prod_{i \in \mathrm{O}}f(y_i)}{m(\mathbf{y})} \prod_{i \in \mathrm{O} } \frac{(1/\sigma)f((y_i-\mathbf{x}_i^T\boldsymbol\beta)/\sigma)}{\sigma^\gamma f(y_i)} \rightarrow 1,
\]
as $\omega \rightarrow \infty$, for any $\boldsymbol\beta \in \re^p, \sigma > 0$, using Result (a) and \eqref{eqn:limit_PDF}. This concludes the proof of Result (b).

We finish with the proof of Result (c). This result is a direct consequence of Result (b) using Scheff\'{e}'s theorem (see \cite{scheffe1947useful}).
\end{proof}

\begin{proof}[Proof of \autoref{thm:2}]
Let us define the vector of unknown model parameters to be $\boldsymbol\theta := (\boldsymbol\beta, s) \in \re^{p + 1}$, where $\sigma$ will now be viewed as a function of $s$ defined through:
  \[
   \sigma = \begin{cases}
    (s + 1)^2 + 2s^2 \qquad \text{if $s \geq 0$,} \cr
    (1 - s)^{-2} \quad \text{if $s < 0$.}
   \end{cases}
  \]
This function is strictly increasing and thus defines a bijection. It is smooth, in the sense that its two first derivatives are continuous. These are technical requirements to prove \autoref{thm:2}.

We now present the proof of Result (a). We have that
\[
 \E[\log p_{0}(Y \mid \mathbf{X})] = - \log \sigma_0 - (2 \sigma_0^2)^{-1} \E[(Y - \mathbf{X}^T \boldsymbol\beta_0)^2] + \textsf{cst} = - \log \sigma_0 - 1 / 2 + \textsf{cst},
\]
where $\textsf{cst}$ will be used to denote a generic constant that does not depend on $\boldsymbol\beta, \sigma, \boldsymbol\beta, \sigma_0$. We omitted the index given that the random variables $\mathbf{Z}_1, \ldots, \mathbf{Z}_n$ are IID.

Also,
\begin{align*}
 \E[\log p_{\boldsymbol\theta}(Y \mid \mathbf{X})] &= \int \mu_\mathbf{X}(\d \mathbf{x}) \int \left(-\log \sigma - \frac{\gamma + 1}{2} \log\left(1 + \frac{(y - \mathbf{x}^T\boldsymbol\beta)^2}{\sigma^2 \gamma}\right)\right) \sigma_0^{-1} g(\sigma_0^{-1}(y - \mathbf{x}^T\boldsymbol\beta_0)) \, \d y + \textsf{cst} \cr
 &= \int \mu_\mathbf{X}(\d \mathbf{x}) \int \left(-\log \sigma - \frac{\gamma + 1}{2} \log\left(1 + \frac{(u - \mathbf{x}^T(\boldsymbol\beta - \boldsymbol\beta_0) / \sigma_0)^2}{(\sigma / \sigma_0)^2 \gamma}\right)\right) g(u) \, \d u + \textsf{cst} \cr
 &= -\eta - \log \sigma_0 - \frac{\gamma + 1}{2} \int \mu_\mathbf{X}(\d \mathbf{x}) \int \log\left(1 + \frac{(u - \mathbf{x}^T \boldsymbol\xi)^2}{\ee^{2 \eta}\gamma}\right) g(u) \, \d u + \textsf{cst} \cr
 &= -\eta - \log \sigma_0 - \frac{\gamma + 1}{2} \int \mu_\mathbf{X}(\d \mathbf{x}) \int \log\left(1 + \frac{(u - \mathbf{x}^T \boldsymbol\xi)^2}{\ee^{2\eta}\gamma}\right) g(u) \, \d u + \textsf{cst},
\end{align*}
using the change of variable $u = \sigma_0^{-1}(y - \mathbf{x}^T\boldsymbol\beta_0)$ and the reparametrization $(\boldsymbol\xi, \eta) := ((\boldsymbol\beta - \boldsymbol\beta_0) / \sigma_0, \log \sigma / \sigma_0)$ that is used specifically for the proof of Result (a). Therefore, the divergence in the new parametrization is given by
\begin{align*}
 K(\boldsymbol\xi, \eta) = \frac{\gamma + 1}{2} \int \mu_\mathbf{X}(\d \mathbf{x}) \int \log\left(1 + \frac{(u - \mathbf{x}^T \boldsymbol\xi)^2}{\ee^{2\eta}\gamma}\right) g(u) \, \d u + \eta + \textsf{cst}.
\end{align*}

We now prove that
\[
 \int \log\left(1 + \frac{(u - \mathbf{x}^T \boldsymbol\xi)^2}{\ee^{2\eta}\gamma}\right) g(u) \, \d u > \int \log\left(1 + \frac{u^2}{\ee^{2\eta}\gamma}\right) g(u) \, \d u,
\]
for any $\boldsymbol\xi \neq \mathbf{0}, \mathbf{x}, \eta$ and $\gamma$, which implies that the divergence is minimized at $\boldsymbol\xi = \mathbf{0}$ for any $\eta$ and $\gamma$. Recall that the explanatory variables are assumed to be continuous, which implies that $\mathbf{x}^T \boldsymbol\xi = 0$ with probability 0 when $\boldsymbol\xi \neq \mathbf{0}$. We proceed by proving that
\[
 \int \left(\log\left(1 + \frac{(u - \mathbf{x}^T \boldsymbol\xi)^2}{\ee^{2\eta}\gamma}\right) - \log\left(1 + \frac{u^2}{\ee^{2\eta}\gamma}\right)\right) g(u) \, \d u > 0.
\]

Let us consider that $\boldsymbol\xi \neq \mathbf{0}$ and $\mathbf{x}^T \boldsymbol\xi < 0$. The proof for the case $\mathbf{x}^T \boldsymbol\xi > 0$ is analogous. We have
\begin{align*}
 &\int \left(\log\left(1 + \frac{(u - \mathbf{x}^T \boldsymbol\xi)^2}{\ee^{2\eta}\gamma}\right) - \log\left(1 + \frac{u^2}{\ee^{2\eta}\gamma}\right)\right) g(u) \, \d u \cr
 &=\quad \int_{\mathbf{x}^T \boldsymbol\xi}^\infty \left(\log\left(1 + \frac{(u - \mathbf{x}^T \boldsymbol\xi)^2}{\ee^{2\eta}\gamma}\right) - \log\left(1 + \frac{u^2}{\ee^{2\eta}\gamma}\right)\right) g(u) \, \d u \cr
 &\qquad - \int_{-\infty}^{\mathbf{x}^T \boldsymbol\xi} \left(\log\left(1 + \frac{u^2}{\ee^{2\eta}\gamma}\right) - \log\left(1 + \frac{(u - \mathbf{x}^T \boldsymbol\xi)^2}{\ee^{2\eta}\gamma}\right)\right) g(u) \, \d u.
\end{align*}

We find an upper bound for the second integral on the RHS:
\begin{align*}
 &\int_{-\infty}^{\mathbf{x}^T \boldsymbol\xi} \left(\log\left(1 + \frac{u^2}{\ee^{2\eta}\gamma}\right) - \log\left(1 + \frac{(u - \mathbf{x}^T \boldsymbol\xi)^2}{\ee^{2\eta}\gamma}\right)\right) g(u) \, \d u \cr
 &\quad \za{=} \int_{-\mathbf{x}^T \boldsymbol\xi}^{\infty} \left(\log\left(1 + \frac{z^2}{\ee^{2\eta}\gamma}\right) - \log\left(1 + \frac{(-z - \mathbf{x}^T \boldsymbol\xi)^2}{\ee^{2\eta}\gamma}\right)\right) g(z) \, \d z \cr
 &\quad \zb{<} \int_{-\mathbf{x}^T \boldsymbol\xi}^{\infty} \left(\log\left(1 + \frac{z^2}{\ee^{2\eta}\gamma}\right) - \log\left(1 + \frac{(z + \mathbf{x}^T \boldsymbol\xi)^2}{\ee^{2\eta}\gamma}\right)\right) g(z + \mathbf{x}^T \boldsymbol\xi) \, \d z \cr
 &\quad \zc{=} \int_{0}^{\infty} \left(\log\left(1 + \frac{(w - \mathbf{x}^T \boldsymbol\xi)^2}{\ee^{2\eta}\gamma}\right) - \log\left(1 + \frac{w^2}{\ee^{2\eta}\gamma}\right)\right) g(w) \, \d w.
\end{align*}
In Step $a$, we use the change of variable $z = -u$. In Step $b$, we use that $z > z + \mathbf{x}^T \boldsymbol\xi \geq 0$, implying that $g(z) < g(z + \mathbf{x}^T \boldsymbol\xi)$. In Step $c$, we use the change of variable $w = z + \mathbf{x}^T \boldsymbol\xi \Leftrightarrow w - \mathbf{x}^T \boldsymbol\xi= z$.

Therefore,
\begin{align*}
 &\int \left(\log\left(1 + \frac{(u - \mathbf{x}^T \boldsymbol\xi)^2}{\ee^{2\eta}\gamma}\right) - \log\left(1 + \frac{u^2}{\ee^{2\eta}\gamma}\right)\right) g(u) \, \d u \cr
 &\quad > \quad \int_{\mathbf{x}^T \boldsymbol\xi}^\infty \left(\log\left(1 + \frac{(u - \mathbf{x}^T \boldsymbol\xi)^2}{\ee^{2\eta}\gamma}\right) - \log\left(1 + \frac{u^2}{\ee^{2\eta}\gamma}\right)\right) g(u) \, \d u \cr
 &\qquad - \int_{0}^{\infty} \left(\log\left(1 + \frac{(w - \mathbf{x}^T \boldsymbol\xi)^2}{\ee^{2\eta}\gamma}\right) - \log\left(1 + \frac{w^2}{\ee^{2\eta}\gamma}\right)\right) g(w) \, \d w \cr
 &\quad = \int_{\mathbf{x}^T \boldsymbol\xi}^0 \left(\log\left(1 + \frac{(u - \mathbf{x}^T \boldsymbol\xi)^2}{\ee^{2\eta}\gamma}\right) - \log\left(1 + \frac{u^2}{\ee^{2\eta}\gamma}\right)\right) g(u) \, \d u.
\end{align*}

We proceed similarly as before to show that this integral is greater than $0$:
\begin{align*}
 &\int_{\mathbf{x}^T \boldsymbol\xi}^0 \left(\log\left(1 + \frac{(u - \mathbf{x}^T \boldsymbol\xi)^2}{\ee^{2\eta}\gamma}\right) - \log\left(1 + \frac{u^2}{\ee^{2\eta}\gamma}\right)\right) g(u) \, \d u \cr
 &\quad = \int_{\mathbf{x}^T \boldsymbol\xi / 2}^{-\mathbf{x}^T \boldsymbol\xi / 2} \left(\log\left(1 + \frac{(t - \mathbf{x}^T \boldsymbol\xi / 2)^2}{\ee^{2\eta}\gamma}\right) - \log\left(1 + \frac{(t + \mathbf{x}^T \boldsymbol\xi / 2)^2}{\ee^{2\eta}\gamma}\right)\right) g(t + \mathbf{x}^T \boldsymbol\xi / 2) \, \d t \cr
 &\quad = \int_{0}^{-\mathbf{x}^T \boldsymbol\xi / 2} \left(\log\left(1 + \frac{(t - \mathbf{x}^T \boldsymbol\xi / 2)^2}{\ee^{2\eta}\gamma}\right) - \log\left(1 + \frac{(t + \mathbf{x}^T \boldsymbol\xi / 2)^2}{\ee^{2\eta}\gamma}\right)\right) g(t + \mathbf{x}^T \boldsymbol\xi / 2) \, \d t \cr
 &\qquad - \int_{\mathbf{x}^T \boldsymbol\xi / 2}^{0} \left(\log\left(1 + \frac{(t + \mathbf{x}^T \boldsymbol\xi / 2)^2}{\ee^{2\eta}\gamma}\right) - \log\left(1 + \frac{(t - \mathbf{x}^T \boldsymbol\xi / 2)^2}{\ee^{2\eta}\gamma}\right)\right) g(t + \mathbf{x}^T \boldsymbol\xi / 2) \, \d t,
\end{align*}
using the change of variable $t = u - \mathbf{x}^T \boldsymbol\xi / 2$. We find an upper bound for the second integral on the RHS:
\begin{align*}
 &\int_{\mathbf{x}^T \boldsymbol\xi / 2}^{0} \left(\log\left(1 + \frac{(t + \mathbf{x}^T \boldsymbol\xi / 2)^2}{\ee^{2\eta}\gamma}\right) - \log\left(1 + \frac{(t - \mathbf{x}^T \boldsymbol\xi / 2)^2}{\ee^{2\eta}\gamma}\right)\right) g(t + \mathbf{x}^T \boldsymbol\xi / 2) \, \d t \cr
 &\quad \za{=} \int_{0}^{-\mathbf{x}^T \boldsymbol\xi / 2} \left(\log\left(1 + \frac{(z - \mathbf{x}^T \boldsymbol\xi / 2)^2}{\ee^{2\eta}\gamma}\right) - \log\left(1 + \frac{(z + \mathbf{x}^T \boldsymbol\xi / 2)^2}{\ee^{2\eta}\gamma}\right)\right) g(z - \mathbf{x}^T \boldsymbol\xi / 2) \, \d z \cr
 &\quad \zb{<} \int_{0}^{-\mathbf{x}^T \boldsymbol\xi / 2} \left(\log\left(1 + \frac{(z - \mathbf{x}^T \boldsymbol\xi / 2)^2}{\ee^{2\eta}\gamma}\right) - \log\left(1 + \frac{(z + \mathbf{x}^T \boldsymbol\xi / 2)^2}{\ee^{2\eta}\gamma}\right)\right) g(z + \mathbf{x}^T \boldsymbol\xi / 2) \, \d z.
\end{align*}
In Step $a$, we use the change of variable $z = -t$. In Step $b$, we use that $(z - \mathbf{x}^T \boldsymbol\xi / 2)^2 > (z + \mathbf{x}^T \boldsymbol\xi / 2)^2$ when $-\mathbf{x}^T \boldsymbol\xi / 2 \geq z > 0$, implying that $g(z - \mathbf{x}^T \boldsymbol\xi / 2) < g(z + \mathbf{x}^T \boldsymbol\xi / 2)$. Indeed,
\[
 (z - \mathbf{x}^T \boldsymbol\xi / 2)^2 = (z + \mathbf{x}^T \boldsymbol\xi / 2 - \mathbf{x}^T \boldsymbol\xi)^2 = (z + \mathbf{x}^T \boldsymbol\xi / 2)^2 - 2 \mathbf{x}^T \boldsymbol\xi z > (z + \mathbf{x}^T \boldsymbol\xi / 2)^2,
\]
because $(\mathbf{x}^T \boldsymbol\xi)^2 \geq - 2 \mathbf{x}^T \boldsymbol\xi z > 0$. This concludes the proof that the divergence is minimized at $\boldsymbol\xi = \mathbf{0}$ for any $\eta$ and $\gamma$.

To minimize the divergence with respect to $\eta$, we thus set $\boldsymbol\xi = \mathbf{0}$ and proceed, meaning that we minimize the following with respect to $\eta$:
\begin{align*}
 K(\mathbf{0}, \eta) &= \frac{\gamma + 1}{2} \int \mu_\mathbf{X}(\d \mathbf{x}) \int \log\left(1 + \frac{u}{\ee^{2\eta}\gamma}\right) g(u) \, \d u + \eta + \textsf{cst} \\
 &=  \frac{\gamma + 1}{2} \int \log\left(1 + \frac{u^2}{\ee^{2\eta}\gamma}\right) g(u) \, \d u + \eta + \textsf{cst}.
\end{align*}

We proceed by differentiating the function with respect to $\eta$:
\[
 \frac{\partial}{\partial \eta} K(\mathbf{0}, \eta) = \frac{\gamma + 1}{2} \frac{\partial}{\partial \eta} \int \log\left(1 + \frac{u^2}{\ee^{2\eta}\gamma}\right) g(u) \, \d u + 1.
\]
We show that we can differentiate under the under the integral (see, e.g., \cite{rosenthal2006first}). First, we need to show the integral is finite for any $\eta$ and $\gamma$. Using Jensen's inequality,
\[
 \int \log\left(1 + \frac{u^2}{\ee^{2\eta}\gamma}\right) g(u) \, \d u \leq \log \int \left(1 + \frac{u^2}{\ee^{2\eta}\gamma}\right) g(u) \, \d u = \log\left(1 + \frac{1}{\ee^{2\eta}\gamma}\right) < \infty.
\]
Second, we show that the derivative of the integrand is bounded, which is sufficient and allows to conclude:
\[
  \left|\frac{\partial}{\partial \eta} \log\left(1 + \frac{u^2}{\ee^{2\eta}\gamma}\right)\right| = \left|\frac{-2 u^2}{\ee^{2\eta}\gamma + u^2}\right| \leq 2,
\]
given that $\ee^{2\eta}\gamma \geq 0$. Therefore,
\begin{align}\label{eqn:der_eta}
 \frac{\partial}{\partial \eta} K(\mathbf{0}, \eta) = \frac{\gamma + 1}{2} \int \frac{-2 u^2}{\ee^{2\eta}\gamma + u^2} \, g(u) \, \d u + 1.
\end{align}

We now show that $K(\mathbf{0}, \eta)$ is minimized at $\eta^*$, which is the solution to
\begin{align}\label{eqn:Thm_2_a_alternative}
 (\gamma + 1) \int \frac{u^2}{\ee^{2\eta}\gamma + u^2} \, g(u) \, \d u - 1 = 0.
\end{align}
For any $\gamma$, the function on the RHS in \eqref{eqn:der_eta} is a continuous and strictly increasing function of $\eta$ that has a minimum of $-\gamma$ and a maximum of $1$, implying that the solution exists and is unique. This also implies that the function $K(\mathbf{0}, \eta)$ is strictly decreasing for $\eta < \eta^*$ and then strictly increasing for $\eta > \eta^*$, which concludes the proof. This function is for instance minimized at $\eta^* = -0.4910$ when $\gamma = 1$, which implies that $\sigma^* / \sigma_0 = 0.6120$ (see \autoref{fig:fct_eta}).

The integral in \eqref{eqn:Thm_2_a_alternative} is equal to
\[
 1 - \sqrt{2 \pi \ee^{2\eta} \gamma} \, \exp\left(\ee^{2\eta} \gamma / 2\right) \left(1 - \Phi\left(\sqrt{\ee^{2\eta} \gamma}\right)\right),
\]
which explains why \eqref{eqn:Thm_2_a_alternative} is equivalent to \eqref{eqn:Thm_2_a}.

   \begin{figure}[ht]
 \centering
   \includegraphics[width=0.50\textwidth]{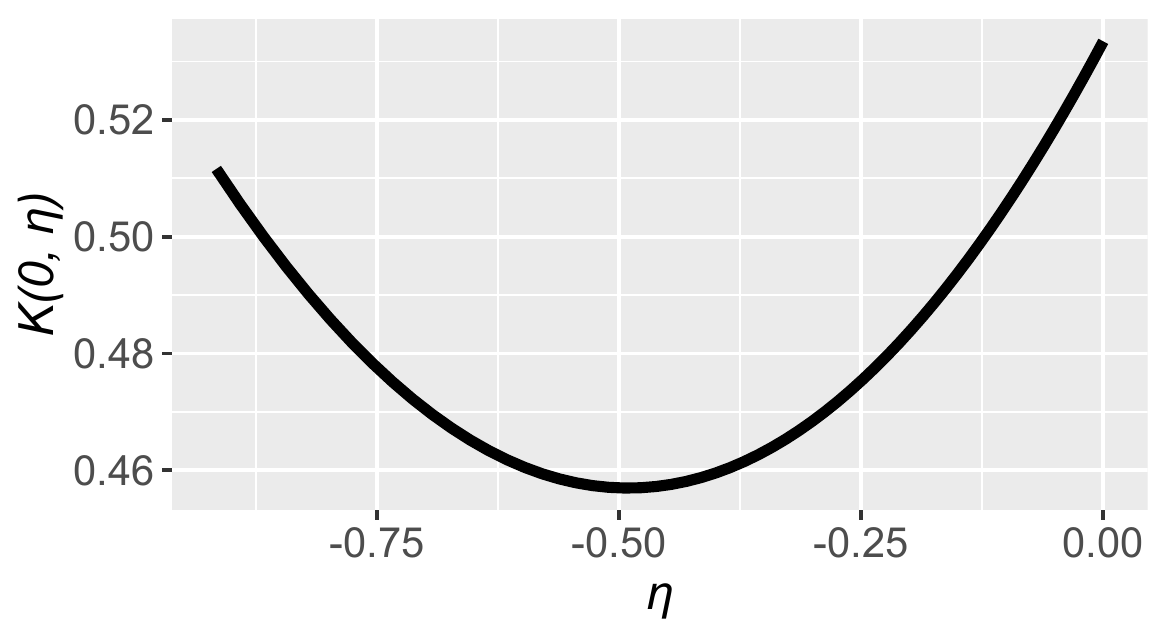}
  \vspace{-3mm}
\caption{\small Function $K(\mathbf{0}, \eta)$ when $\gamma = 1$}\label{fig:fct_eta}
\end{figure}
\normalsize

We turn to the proof of Result (b). To prove this result, we verify the assumptions of Theorems 2.1 and 2.2 in \cite{bunke1998asymptotic}.

\textbf{A1}: \textit{the parameter space $\boldsymbol\Theta = \re^{p + 1}$ is a closed (possibly unbounded) convex set in $\re^{p + 1}$ with a non-empty interior, the density $p_{\boldsymbol\theta}( \, \cdot \mid \cdot \,) f_\mathbf{X}(\, \cdot \,)$ is bounded for all $\boldsymbol\theta$ and $(y, \mathbf{x})$, and its carrier $\{(y, \mathbf{x}): p_{\boldsymbol\theta}(y \mid \mathbf{x}) f_\mathbf{X}(\mathbf{x}) > 0\}$ is the same for all $\boldsymbol\theta$.}

This assumption is seen to be verified.

\textbf{A2}: \textit{for all $\boldsymbol\theta$, there is a sphere $S[\boldsymbol\theta, r]$ of center $\boldsymbol\theta$ and radius $r > 0$ which may depend on $\boldsymbol\theta$ with}
	 \[
	  \E\left[\sup\left\{\left|\log\frac{p_{0}(Y \mid \mathbf{X}) \, f_\mathbf{X}(\mathbf{X})}{p_{(\mathbf{t}, w)}(Y \mid \mathbf{X}) \, f_\mathbf{X}(\mathbf{X})}\right|: (\mathbf{t}, w) \in S[\boldsymbol\theta, r]\right\}\right] = \E\left[\sup\left\{\left|\log\frac{p_{0}(Y \mid \mathbf{X})}{p_{(\mathbf{t}, w)}(Y \mid \mathbf{X})}\right|: (\mathbf{t}, w) \in S[\boldsymbol\theta, r]\right\}\right] < \infty.
	 \]

 For fixed $(y, \mathbf{x})$, we find an upper bound for
 \[
  \left|\log\frac{p_{0}(y \mid \mathbf{x})}{p_{(\mathbf{t}, w)}(y \mid \mathbf{x})}\right|
 \]
 that does not depend on $(\mathbf{t}, w)$ and that is valid for any $(y, \mathbf{x})$ using that $(\mathbf{t}, w) \in S[\boldsymbol\theta, r]$. Using the triangle inequality,
\[
 \left|\log\frac{p_{0}(y \mid \mathbf{x})}{p_{(\mathbf{t}, w)}(y \mid \mathbf{x})}\right| \leq |\log p_{0}(y \mid \mathbf{x})| + |\log p_{(\mathbf{t}, w)}(y \mid \mathbf{x})|.
\]
The first term on the RHS does not depend on $(\mathbf{t}, w)$; we can thus focus on the second term.

Using the triangle inequality again and writing the scale parameter as $\sigma$ to simplify (which here is a function of $w$),
\begin{align*}
 |\log p_{(\mathbf{t}, w)}(y \mid \mathbf{x})| &= \left|-\log \sigma - \frac{\gamma + 1}{2}\log\left(1 + \frac{(y - \mathbf{x}^T \mathbf{t})^2}{\sigma^2 \gamma}\right)\right| \cr
 &\leq \log \sigma + \frac{\gamma + 1}{2}\log\left(1 + \frac{(y - \mathbf{x}^T \mathbf{t})^2}{\sigma^2 \gamma}\right) \cr
 &\leq \log \overline{\sigma} + \frac{\gamma + 1}{2}\log\left(1 + \frac{(y - \mathbf{x}^T \mathbf{t})^2}{\underline{\sigma}^2 \gamma}\right),
\end{align*}
where $\overline{\sigma}$ and $\underline{\sigma}$ are the minimum and maximum of $\sigma$ (viewed as a function of $w$) in $S[\boldsymbol\theta, r]$, respectively. We have the following upper bound on $(y - \mathbf{x}^T \mathbf{t})^2$ using that $(a + b)^2 \leq 2a^2 + 2b^2$ for any $a,b \in \re$,
\begin{align*}
 (y - \mathbf{x}^T \mathbf{t})^2 = (y - \mathbf{x}^T \boldsymbol\beta + \mathbf{x}^T (\boldsymbol\beta - \mathbf{t}))^2 &\leq 2 (y - \mathbf{x}^T \boldsymbol\beta)^2 + 2( \mathbf{x}^T (\boldsymbol\beta - \mathbf{t}))^2 \cr
 &\leq 4y^2 + 4(\mathbf{x}^T \boldsymbol\beta)^2 + 2( \mathbf{x}^T (\boldsymbol\beta - \mathbf{t}))^2,
\end{align*}
and $\mathbf{x}^T (\boldsymbol\beta - \mathbf{t}) \leq \|\mathbf{x}\| \, \|\boldsymbol\beta - \mathbf{t}\| \leq \|\mathbf{x}\| r$, where $\| \, \cdot \,\|$ is the Euclidean norm.

Therefore, for any $(y, \mathbf{x})$
\[
 \left|\log\frac{p_{0}(y \mid \mathbf{x})}{p_{(\mathbf{t}, w)}(y \mid \mathbf{x})}\right| \leq  |\log p_{0}(y \mid \mathbf{x})| + \log\overline{\sigma} + \frac{\gamma + 1}{2}\log\left(1 + \frac{4 y^2 + (4\|\boldsymbol\beta\|^2 + 2 r^2)\|\mathbf{x}\|^2}{\underline{\sigma}^2 \gamma}\right),
\]
implying that
\begin{align*}
 &\E\left[\sup\left\{\left|\log\frac{p_{0}(Y \mid \mathbf{X})}{p_{(\mathbf{t}, w)}(Y \mid \mathbf{X})}\right|: (\mathbf{t}, w) \in S[\boldsymbol\theta, r]\right\}\right] \cr
  &\quad\leq \E\left[|\log p_{0}(Y \mid \mathbf{X})| + \log\overline{\sigma} + \frac{\gamma + 1}{2}\log\left(1 + \frac{4 Y^2 + (4\|\boldsymbol\beta\|^2 + 2 r^2)\|\mathbf{X}\|^2}{\underline{\sigma}^2 \gamma}\right)\right].
\end{align*}

Now we prove that the expectation on the RHS is finite. Using the triangle inequality,
\[
 \E[|\log p_{0}(Y \mid \mathbf{X})|] \leq \log \sigma_0 + (2\sigma_0^2)^{-1} \E[(Y - \mathbf{X}^T \boldsymbol\beta_0)^2] + \textsf{cst} = \log \sigma_0 + 1 / 2 + \textsf{cst}.
\]

Using Jensen's inequality, that $Y = \mathbf{X}^T \boldsymbol\beta_0 + \sigma_0 \varepsilon$ with $\varepsilon \sim g$, that $(a + b)^2 \leq 2a^2 + 2b^2$ for any $a,b \in \re$, and the Cauchy--Schwarz inequality,
\begin{align*}
 \E\left[\log\left(1 + \frac{4 Y^2 + (4\|\boldsymbol\beta\|^2 + 2 r^2)\|\mathbf{X}\|^2}{\underline{\sigma}^2 \gamma}\right)\right] &\leq \log\left(1 + \frac{\E[4 Y^2 + (4\|\boldsymbol\beta\|^2 + 2 r^2)\|\mathbf{X}\|^2]}{\underline{\sigma}^2 \gamma}\right) \cr
 &=  \log\left(1 + \frac{\E[4(\mathbf{X}^T \boldsymbol\beta_0 + \sigma_0 \epsilon)^2 + (4\|\boldsymbol\beta\|^2 + 2 r^2)\|\mathbf{X}\|^2]}{\underline{\sigma}^2 \gamma}\right) \cr
 &\leq \log\left(1 + \frac{\E[8(\mathbf{X}^T \boldsymbol\beta_0)^2 + 8(\sigma_0 \epsilon)^2 + (4\|\boldsymbol\beta\|^2 + 2 r^2)\|\mathbf{X}\|^2]}{\underline{\sigma}^2 \gamma}\right) \cr
 &\leq \log\left(1 + \frac{8\sigma_0^2 + (8\|\boldsymbol\beta_0\|^2 + 4\|\boldsymbol\beta\|^2 + 2 r^2) \E[\|\mathbf{X}\|^2]}{\underline{\sigma}^2 \gamma}\right),
\end{align*}
which is finite given that $\E[\|\mathbf{X}\|^2] < \infty$.

\textbf{A3}: \textit{for all fixed $(y, \mathbf{x})$, the density $p_{\boldsymbol\theta}(y \mid \mathbf{x}) f_\mathbf{X}(\mathbf{x})$ has a continuous derivative $p_{\boldsymbol\theta}'(y \mid \mathbf{x}) f_\mathbf{X}(\mathbf{x})$ with respect to $\boldsymbol\theta$ and there are positive constants $c, b_0$ such that
	 \[
	  \int \mu_\mathbf{X}(\d \mathbf{x}) \int \| [p_{\boldsymbol\theta}(y \mid \mathbf{x}) ]^{-1} p_{\boldsymbol\theta}'(y \mid \mathbf{x})  \|^{4(p+2)} \, p_{\boldsymbol\theta}(y \mid \mathbf{x}) \, \d y < c\, (1+\|\boldsymbol\theta\|^{b_0}),
	 \]
	 for all $\boldsymbol\theta$, where $\|\cdot\|$ denotes a norm in $\re^{p + 1}$.}

  We have that
  \[
   \sigma = \begin{cases}
    (s + 1)^2 + 2s^2 \qquad \text{if $s \geq 0$,} \cr
    (1 - s)^{-2} \quad \text{if $s < 0$.}
   \end{cases}
  \]
  Therefore,
  \[
   \frac{\partial \sigma}{\partial s} = \begin{cases}
    2(s + 1) + 4s \qquad \text{if $s \geq 0$,} \cr
    2(1 - s)^{-3} \quad \text{if $s < 0$.}
   \end{cases}
  \]
  We use $\sigma_s$ to denote the latter derivative, viewed as a function of $s$.

  In our case,
  \[
   [p_{\boldsymbol\theta}(y \mid \mathbf{x}) ]^{-1} p_{\boldsymbol\theta}'(y \mid \mathbf{x}) = \left(\begin{array}{c}
    \sigma^{-1} \frac{\gamma + 1}{\gamma}\left(1 + \frac{(y - \mathbf{x}^T \boldsymbol\beta)^2}{\sigma^2 \gamma}\right)^{-1} \frac{y - \mathbf{x}^T \boldsymbol\beta}{\sigma}  \, \mathbf{x} \cr
    - \sigma^{-1} \sigma_s + \sigma^{-1} \sigma_s \frac{\gamma + 1}{\gamma} \left(1 + \frac{(y - \mathbf{x}^T \boldsymbol\beta)^2}{\sigma^2 \gamma}\right)^{-1} \frac{(y - \mathbf{x}^T \boldsymbol\beta)^2}{\sigma^2}
    \end{array}
   \right).
  \]

  Taking the Euclidean norm, we have
  \begin{align*}
   \|[p_{\boldsymbol\theta}(y \mid \mathbf{x}) ]^{-1} p_{\boldsymbol\theta}'(y \mid \mathbf{x})\|^2 &\leq \sigma^{-2} \left(\left(\frac{\gamma + 1}{\gamma}\right)^2 \|\mathbf{x}\|^2 + \sigma_s^2\left(1 + \left(\frac{\gamma + 1}{\gamma}\right)\right)^2\right) \cr
   &\leq 36 (|s| + 1)^4 \left(\left(\frac{\gamma + 1}{\gamma}\right)^2 \|\mathbf{x}\|^2 + \left(1 + \left(\frac{\gamma + 1}{\gamma}\right)\right)^2\right),
  \end{align*}
  using that
  \[
   0 \leq \left(1 + \frac{(y - \mathbf{x}^T \boldsymbol\beta)^2}{\sigma^2 \gamma}\right)^{-1} \left|\frac{y - \mathbf{x}^T \boldsymbol\beta}{\sigma}\right| \leq 1,
  \]
  \[
   0 \leq \left(1 + \frac{(y - \mathbf{x}^T \boldsymbol\beta)^2}{\sigma^2 \gamma}\right)^{-1} \frac{(y - \mathbf{x}^T \boldsymbol\beta)^2}{\sigma^2} \leq 1,
  \]
  and
  \[
   1, \sigma_s \leq 6(|s| + 1) \quad \text{and} \quad \sigma^{-1} \leq (|s| + 1)^2.
  \]
  Therefore,
  \begin{align*}
   &\int \mu_\mathbf{X}(\d \mathbf{x}) \int \| [p_{\boldsymbol\theta}(y \mid \mathbf{x}) ]^{-1} p_{\boldsymbol\theta}'(y \mid \mathbf{x})  \|^{4(p+2)} \, p_{\boldsymbol\theta}(y \mid \mathbf{x}) \, \d y \cr
   &\qquad \leq 6^{4(p + 2)}(|s| + 1)^{8(p+2)} \int \mu_\mathbf{X}(\d \mathbf{x}) \left(\left(\frac{\gamma + 1}{\gamma}\right)^2 \|\mathbf{x}\|^2 + \left(1 + \left(\frac{\gamma + 1}{\gamma}\right)\right)^2\right)^{2(p + 2)}.
  \end{align*}
  The integral is finite given that $\E\|\mathbf{X}\|^{4(p + 2)} < \infty$. Also, there exists a positive constant $c$ such that $(|s| + 1)^{8(p+2)}$ is upper bounded by $c(1 + \|s\|^{8(p+2)})$, which allows to conclude that A3 is verified.

\textbf{A4}: \textit{for some positive constant $b_1$, the affinity has the following behaviour:
\[
 \int \mu_\mathbf{X}(\d \mathbf{x}) \int \left[p_{\boldsymbol\theta}(y \mid \mathbf{x})p_{0}(y \mid \mathbf{x})\right]^{1/2} \d y < c \|\boldsymbol\theta\|^{-b_1}, \quad \boldsymbol\theta \in \re^{p + 1}.
\]
}

First, we show that the affinity has this behaviour when the parameters are considered to be in a compact set such that $|s|, |\beta_1|, \ldots, |\beta_p| \leq k$, where $k$ is a positive constant. Using the Cauchy--Schwarz inequality,
\[
 \int \mu_\mathbf{X}(\d \mathbf{x}) \int \left[p_{\boldsymbol\theta}(y \mid \mathbf{x})p_{0}(y \mid \mathbf{x})\right]^{1/2} \d y \leq 1.
\]
Therefore, if $\|\boldsymbol\theta\| \leq 1$, then
\[
 \int \mu_\mathbf{X}(\d \mathbf{x}) \int \left[p_{\boldsymbol\theta}(y \mid \mathbf{x})p_{0}(y \mid \mathbf{x})\right]^{1/2} \d y \leq \|\boldsymbol\theta\|^{-b_1}
\]
for any $b_1$. If $\|\boldsymbol\theta\| > 1$, but $\boldsymbol\theta$ belongs to the compact set,
\[
 \int \mu_\mathbf{X}(\d \mathbf{x}) \int \left[p_{\boldsymbol\theta}(y \mid \mathbf{x})p_{0}(y \mid \mathbf{x})\right]^{1/2} \d y \leq \frac{\|\boldsymbol\theta\|^{b_1}}{\|\boldsymbol\theta\|^{b_1}} \leq \frac{(p + 1)^{b_1/2} k^{b_1}}{\|\boldsymbol\theta\|^{b_1}}
\]
for any $b_1$.

We can thus focus on the case where the parameters are outside of the compact set and are such that $|s|, |\beta_1|, \ldots, |\beta_p| > k$. We have that
\begin{align*}
   \int \left[p_{\boldsymbol\theta}(y \mid \mathbf{x})p_{0}(y \mid \mathbf{x})\right]^{1/2} \d y &= \int \frac{1}{\sigma^{1/2}} f^{1/2}\left(\frac{y - \mathbf{x}^T\boldsymbol\beta}{\sigma}\right) \, \frac{1}{\sigma_0^{1/2}} g^{1/2}\left(\frac{y - \mathbf{x}^T\boldsymbol\beta_0}{\sigma_0}\right) \, \d y \cr
   &\za{=} \sigma_0^{1/2} \int \frac{1}{\sigma^{1/2}} f^{1/2}\left(\frac{u - \mathbf{x}^T(\boldsymbol\beta - \boldsymbol\beta_0) / \sigma_0}{\sigma / \sigma_0}\right) \, g^{1/2}\left(u \right) \, \d u \cr
   &\zb{\leq} \sigma_0\max\{1, \sigma_0^{-\frac{\gamma + 1}{2}}\} \int \frac{1}{\sigma^{1/2}} f^{1/2}\left(\frac{u - \mathbf{x}^T(\boldsymbol\beta - \boldsymbol\beta_0) / \sigma_0}{\sigma}\right) \, g^{1/2}(u) \, \d u \cr
   &\propto \int \frac{1}{\sigma^{1/2}} \left(1 + \frac{(u - \mathbf{x}^T(\boldsymbol\beta - \boldsymbol\beta_0) / \sigma_0)^2}{\sigma^2 \gamma}\right)^{-\frac{\gamma + 1}{4}} \, g^{1/2}(u) \, \d u \cr
   &\zc{\leq} \gamma^{\frac{\gamma + 1}{4}} \int \frac{1}{\sigma^{1/2}} \left(1 + \frac{(u - \mathbf{x}^T(\boldsymbol\beta - \boldsymbol\beta_0) / \sigma_0)^2}{\sigma^2}\right)^{-\frac{1}{2}} \, g^{1/2}(u) \, \d u \cr
   &\zd{<} \gamma^{\frac{\gamma + 1}{4}} \sigma_0^{1/32} |\mathbf{x}^T(\boldsymbol\beta - \boldsymbol\beta_0)|^{-1/32} \, \textsf{cst} \times \begin{cases}
    (|s| + 1)^{-1/16} \quad \text{if $s \geq 0$,} \cr
    (|s| + 1)^{-1/2} \quad \text{if $s < 0$.}
    \end{cases} \cr
    &\leq \gamma^{\frac{\gamma + 1}{4}} \sigma_0^{1/32} |\mathbf{x}^T(\boldsymbol\beta - \boldsymbol\beta_0)|^{-1/32} \, \textsf{cst} \, (|s| + 1)^{-1/2}.
\end{align*}
In Step $a$, we used the change of variable $u = \sigma_0^{-1}(y - \mathbf{x}^T\boldsymbol\beta_0)$. In Step $b$, we used that
\[
 f^{1/2}\left(\frac{u - \mathbf{x}^T(\boldsymbol\beta - \boldsymbol\beta_0) / \sigma_0}{\sigma / \sigma_0}\right) \leq f^{1/2}\left(\frac{u - \mathbf{x}^T(\boldsymbol\beta - \boldsymbol\beta_0) / \sigma_0}{\sigma}\right) \, \max\{1, \sigma_0^{-\frac{\gamma + 1}{2}}\}.
\]
Indeed, if $\sigma_0 > 1$, $(\sigma / \sigma_0)^{-1}(u - \mathbf{x}^T(\boldsymbol\beta - \boldsymbol\beta_0) / \sigma_0) > \sigma^{-1}(u - \mathbf{x}^T(\boldsymbol\beta - \boldsymbol\beta_0) / \sigma_0)$, then by monotonicity of $f$,
\[
 f^{1/2}\left(\frac{u - \mathbf{x}^T(\boldsymbol\beta - \boldsymbol\beta_0) / \sigma_0}{\sigma / \sigma_0}\right) \leq f^{1/2}\left(\frac{u - \mathbf{x}^T(\boldsymbol\beta - \boldsymbol\beta_0) / \sigma_0}{\sigma}\right) =  f^{1/2}\left(\frac{u - \mathbf{x}^T(\boldsymbol\beta - \boldsymbol\beta_0) / \sigma_0}{\sigma}\right) \, \max\{1, \sigma_0^{-\frac{\gamma + 1}{2}}\}.
\]
If $\sigma_0 \leq 1$, we apply \autoref{lemma:bound_student} to obtain
\begin{align*}
 f^{1/2}\left(\frac{u - \mathbf{x}^T(\boldsymbol\beta - \boldsymbol\beta_0) / \sigma_0}{\sigma / \sigma_0}\right) &\leq f^{1/2}\left(\frac{u - \mathbf{x}^T(\boldsymbol\beta - \boldsymbol\beta_0) / \sigma_0}{\sigma}\right) \, \sigma_0^{-\frac{\gamma + 1}{2}} \cr
 &=  f^{1/2}\left(\frac{u - \mathbf{x}^T(\boldsymbol\beta - \boldsymbol\beta_0) / \sigma_0}{\sigma}\right) \, \max(1, \sigma_0^{-\frac{\gamma + 1}{2}}).
\end{align*}
In Step $c$, we used that
\[
 \frac{1}{\gamma} + \frac{(u - \mathbf{x}^T(\boldsymbol\beta - \boldsymbol\beta_0) / \sigma_0)^2}{\sigma^2 \gamma} \leq 1 + \frac{(u - \mathbf{x}^T(\boldsymbol\beta - \boldsymbol\beta_0) / \sigma_0)^2}{\sigma^2 \gamma},
\]
because $\gamma \geq 1$, implying that
\begin{align*}
 \left(1 + \frac{(u - \mathbf{x}^T(\boldsymbol\beta - \boldsymbol\beta_0) / \sigma_0)^2}{\sigma^2 \gamma}\right)^{-\frac{\gamma + 1}{4}} & \leq \left(\frac{1}{\gamma} + \frac{(u - \mathbf{x}^T(\boldsymbol\beta - \boldsymbol\beta_0) / \sigma_0)^2}{\sigma^2 \gamma}\right)^{-\frac{\gamma + 1}{4}} \cr
 &= \gamma^{\frac{\gamma + 1}{4}} \left(1 + \frac{(u - \mathbf{x}^T(\boldsymbol\beta - \boldsymbol\beta_0) / \sigma_0)^2}{\sigma^2}\right)^{-\frac{\gamma + 1}{4}}.
\end{align*}
We also used that
\[
 \left(1 + \frac{(u - \mathbf{x}^T(\boldsymbol\beta - \boldsymbol\beta_0) / \sigma_0)^2}{\sigma^2}\right)^{-\frac{\gamma + 1}{4}} \leq \left(1 + \frac{(u - \mathbf{x}^T(\boldsymbol\beta - \boldsymbol\beta_0) / \sigma_0)^2}{\sigma^2}\right)^{-\frac{1}{2}}.
\]
In Step $d$, we used Lemma A.9 in \cite{bunke1998asymptotic}.

We assumed that
\[
 \|\boldsymbol\beta - \boldsymbol\beta_0\|^{1/32} \E|\mathbf{X}^T(\boldsymbol\beta - \boldsymbol\beta_0)|^{-1/32} < \infty.
\]
We now show that this allows to conclude that A4 is verified. Note that the assumption above is verified when, for instance, the components in $\mathbf{X}$ are independent normal random variables.

What we prove is that
\[
 \|\boldsymbol\beta - \boldsymbol\beta_0\|^{-1/32} (|s| + 1)^{-1/2}
\]
is bounded above by a constant times $\|\boldsymbol\theta\|^{-1/2}$. We consider that $k$ has been chosen such that there exists a constant $0 < k_0 < 1$ such that $(\beta_i - \beta_{0, i})^2 \geq k_0^2 \beta_i^2$. Therefore,
\[
 \|\boldsymbol\beta - \boldsymbol\beta_0\|^{-1/32} (|s| + 1)^{-1/2} \leq k_0^{-1/32}\|\boldsymbol\beta\|^{-1/32} (|s| + 1)^{-1/2} \leq k_0^{-1/32}\|\boldsymbol\beta\|^{-1/2} (|s| + 1)^{-1/2},
\]
using that $\|\boldsymbol\beta\| \geq 1$. Also
\begin{align*}
 [\|\boldsymbol\beta\|^2 (|s| + 1)^2]^{-1/4} = [\|\boldsymbol\beta\|^2 (s^2 + 2 |s| + 1)]^{-1/4} & \leq  [\|\boldsymbol\beta\|^2 (s^2 + 1)]^{-1/4} \cr
 & = [\|\boldsymbol\beta\|^2 s^2 + \|\boldsymbol\beta\|^2]^{-1/4} \cr
 & \leq [s^2 + \|\boldsymbol\beta\|^2]^{-1/4},
\end{align*}
using again that $\|\boldsymbol\beta\| \geq 1$. This allows to conclude that A4 is verified.

\textbf{A5}: \textit{there are positive constants $b_2, b_3$ such that for all $\boldsymbol\theta \in \re^{p + 1}$ and $r > 0$ it holds that the measure of $S[\boldsymbol\theta, r]$ under the prior distribution is bounded above by $c r^{b_2}(1 + (\|\boldsymbol\theta\| + r)^{b_3})$. Moreover, the measure of $S[\boldsymbol\theta, r]$ under the prior distribution is strictly positive.
}

We assumed that the prior density is strictly positive which implies that the measure of $S[\boldsymbol\theta, r]$ under the prior distribution is strictly positive. Also, if we consider that the prior density is given by $\pi(\, \cdot \,, \cdot \,)$ when using the parametrization $(\boldsymbol\beta, \sigma)$, then, when evaluated at $(\boldsymbol\beta, s)$, it is equal to $\pi(\boldsymbol\beta, \sigma) \, \sigma_s$ under the parametrization $(\boldsymbol\beta, s)$ where in $\pi$, $\sigma$ is viewed as a function of $s$. We assumed under the parametrization $(\boldsymbol\beta, \sigma)$ that $\pi$ is such that
\[
 \pi(\boldsymbol\beta, \sigma) \leq \begin{cases}
  C \quad \text{if $\sigma \geq 1$,} \cr
  C/\sigma \quad \text{if $\sigma < 1$,}
 \end{cases}
\]
which corresponds to an upper bound on the prior density under the parametrization $(\boldsymbol\beta, s)$ given by
\[
 \begin{cases}
  C \, \sigma_s \quad \text{if $s \geq 0$,} \cr
  C (|s| + 1)^2 \, \sigma_s \quad \text{if $s < 0$.}
 \end{cases}
\]
This upper bound is bounded above by $C (|s| + 1)^2 \, \sigma_s \leq 6 C  (|s| + 1)^3$. Therefore, the measure of $S[\boldsymbol\theta, r]$ under the prior distribution is bounded above by
\[
 6C \int_{S[\boldsymbol\theta, r]} (|s| + 1)^3 \, \d \boldsymbol\beta \, \d s \leq 6C  (r + 1)^3 \int_{S[\boldsymbol\theta, r]} \d \boldsymbol\beta \, \d s.
\]
A5 is thus verified.

\textbf{A6}: \textit{let $L: \re^{p + 1} \times \re^{p + 1} \to \re^+$ be a measurable loss function with $L(\boldsymbol\theta, \boldsymbol\theta) = 0$, $c_1, c_2, c_3, b_4, b_5$ be positive constants with
\[
 (c_1 \|\mathbf{t} - \boldsymbol\theta\|^{b_4}) \wedge c_2 \leq L(\mathbf{t}, \boldsymbol\theta) \leq c_3 \|\mathbf{t} - \boldsymbol\theta\|^{b_5},
\]
for all $\mathbf{t}, \boldsymbol\theta$.
}

The quadratic loss is seen to satisfy this assumption.

\textbf{A7}: \textit{the pseudo-true value $\boldsymbol\theta^*$ is unique and belongs to the interior of $\re^{p + 1}$.}

This assumption is verified. This concludes the proof of Result (b).

We turn to the proof of Result (c). To prove this result, we verify the rest of the assumptions of Theorem 4.1 in \cite{bunke1998asymptotic}.

\textbf{A8}: \textit{the function $l(\mathbf{z}, \boldsymbol\theta) := \log[p_{\boldsymbol\theta}(y \mid \mathbf{x}) / p_{0}(y \mid \mathbf{x})]$ has for fixed $\mathbf{z}$ continuous derivatives of second-order with respect to $\boldsymbol\theta$ in the interior of $\re^{p + 1}$,
\[
 l'(\mathbf{z}, \boldsymbol\theta) := \frac{\partial}{\partial \boldsymbol\theta}  l(\mathbf{z}, \boldsymbol\theta), \quad l''(\mathbf{z}, \boldsymbol\theta) := \frac{\partial^2}{\partial \boldsymbol\theta \partial \boldsymbol\theta^T}  l(\mathbf{z}, \boldsymbol\theta).
\]
Moreover, there is a positive function $H$ on $\re^{p + 1}$ and a positive integer $b_6$ with $\E[H(\mathbf{Z})] < \infty$,
\[
 \|l''(\mathbf{z}, \boldsymbol\theta_1) - l''(\mathbf{z}, \boldsymbol\theta_2)\| \leq H(\mathbf{z}) [1 + \|\boldsymbol\theta_1\|^{b_6} + \|\boldsymbol\theta_2\|^{b_6}] \, \|\boldsymbol\theta_1 - \boldsymbol\theta_2\|,
\]
\[
 \|l''(\mathbf{z}, \boldsymbol\theta_1)\| \leq H(\mathbf{z}) [1 + \|\boldsymbol\theta_1\|^{b_6 + 1}],
\]
for any $\boldsymbol\theta_1, \boldsymbol\theta_2$ in the interior of $\re^{p + 1}$, $\| \, \cdot \, \|$ denoting the Euclidean norm on $\re^{p + 1}$ or the analogous norm on the set of $(p + 1) \times (p + 1)$ matrices.
}

We note that the matrix norm that will be used is the Frobenius norm. The first-order derivatives have already been computed (see A3):
\[
   l'(\mathbf{z}, \boldsymbol\theta) = [p_{\boldsymbol\theta}(y \mid \mathbf{x}) ]^{-1} p_{\boldsymbol\theta}'(y \mid \mathbf{x}) = \left(\begin{array}{c}
    \sigma^{-1} \frac{\gamma + 1}{\gamma}\left(1 + \frac{(y - \mathbf{x}^T \boldsymbol\beta)^2}{\sigma^2 \gamma}\right)^{-1} \frac{y - \mathbf{x}^T \boldsymbol\beta}{\sigma}  \, \mathbf{x} \cr
    - \sigma^{-1} \sigma_s + \sigma^{-1} \sigma_s \frac{\gamma + 1}{\gamma} \left(1 + \frac{(y - \mathbf{x}^T \boldsymbol\beta)^2}{\sigma^2 \gamma}\right)^{-1} \frac{(y - \mathbf{x}^T \boldsymbol\beta)^2}{\sigma^2}
    \end{array}
   \right),
  \]
which is a continuous function of $\boldsymbol\theta$.

The matrix of second derivatives is given by
\[
 l''(\mathbf{z}, \boldsymbol\theta) = \left(\begin{array}{cc}
    A & B \cr
    B^T & D
    \end{array}
   \right),
\]
where
\[
 A := -\sigma^{-2} \frac{\gamma + 1}{\gamma} \frac{1 - \frac{(y - \mathbf{x}^T \boldsymbol\beta)^2}{\sigma^2 \gamma}}{\left(1 + \frac{(y - \mathbf{x}^T \boldsymbol\beta)^2}{\sigma^2 \gamma}\right)^2} \, \mathbf{x} \mathbf{x}^T,
\]
\[
 B := -2\sigma^{-2}\sigma_s \frac{\gamma + 1}{\gamma} \frac{\frac{y - \mathbf{x}^T \boldsymbol\beta}{\sigma}}{\left(1 + \frac{(y - \mathbf{x}^T \boldsymbol\beta)^2}{\sigma^2 \gamma}\right)^2} \, \mathbf{x},
\]
\begin{align*}
 D &:= -\sigma^{-1} \sigma_{ss}\left(1- (\gamma + 1) \frac{\frac{(y - \mathbf{x}^T \boldsymbol\beta)^2}{\sigma^2 \gamma}}{\left(1 + \frac{(y - \mathbf{x}^T \boldsymbol\beta)^2}{\sigma^2 \gamma}\right)}\right) \cr
 &\qquad -\sigma^{-2}\sigma_s^2\left(2(\gamma + 1) \frac{\frac{(y - \mathbf{x}^T \boldsymbol\beta)^2}{\sigma^2 \gamma}}{\left(1 + \frac{(y - \mathbf{x}^T \boldsymbol\beta)^2}{\sigma^2 \gamma}\right)^2} + (\gamma + 1) \frac{\frac{(y - \mathbf{x}^T \boldsymbol\beta)^2}{\sigma^2 \gamma}}{\left(1 + \frac{(y - \mathbf{x}^T \boldsymbol\beta)^2}{\sigma^2 \gamma}\right)} - 1\right),
\end{align*}
$\sigma_{ss}$ being the second derivative of $\sigma$ (viewed as a function of $s$) with respect to $s$:
  \[
   \frac{\partial^2 \sigma}{\partial s^2} = \begin{cases}
    6 \qquad \text{if $s \geq 0$,} \cr
    6(1 - s)^{-4} \quad \text{if $s < 0$.}
   \end{cases}
  \]

  $l''(\mathbf{z}, \boldsymbol\theta)$ is a continuous function of $\boldsymbol\theta$. We now show that we are able to obtain a bound of the form:
   \[
 \|l''(\mathbf{z}, \boldsymbol\theta_1) - l''(\mathbf{z}, \boldsymbol\theta_2)\|^2 \leq H^2(\mathbf{z}) [1 + \|\boldsymbol\theta_1\|^{b_6} + \|\boldsymbol\theta_2\|^{b_6}]^2 \, \|\boldsymbol\theta_1 - \boldsymbol\theta_2\|^2,
\]
for the part of the matrix associated with Block $A$. The bounds associated to the other blocks can be obtained similarly and they can all be put together to obtain the desired bound. We have
\begin{align*}
 &\left\|-\sigma_1^{-2} \frac{\gamma + 1}{\gamma} \phi(\boldsymbol\beta_1) \, \mathbf{x} \mathbf{x}^T + \sigma_2^{-2} \frac{\gamma + 1}{\gamma} \phi(\boldsymbol\beta_2) \, \mathbf{x} \mathbf{x}^T\right\|^2
 \cr
&\qquad =\left(\frac{\gamma + 1}{\gamma}\right)^2 \|\mathbf{x} \mathbf{x}^T\|^2(\sigma_1^{-2} \phi(\boldsymbol\beta_1) - \sigma_2^{-2} \phi(\boldsymbol\beta_1) + \sigma_2^{-2} \phi(\boldsymbol\beta_1) - \sigma_2^{-2} \phi(\boldsymbol\beta_2))^2 \cr
&\qquad \leq 2\left(\frac{\gamma + 1}{\gamma}\right)^2 \|\mathbf{x} \mathbf{x}^T\|^2[(\sigma_1^{-2} \phi(\boldsymbol\beta_1) - \sigma_2^{-2} \phi(\boldsymbol\beta_1))^2 + (\sigma_2^{-2} \phi(\boldsymbol\beta_1) - \sigma_2^{-2} \phi(\boldsymbol\beta_2))^2],
\end{align*}
using that for any real numbers $a,b$, $(a + b^2) \leq 2a^2 + 2b^2$ and
where
\[
 \phi(\boldsymbol\beta) := \frac{1 - \frac{(y - \mathbf{x}^T \boldsymbol\beta)^2}{\sigma^2 \gamma}}{\left(1 + \frac{(y - \mathbf{x}^T \boldsymbol\beta)^2}{\sigma^2 \gamma}\right)^2}.
\]

We now look at the first term in the squared brackets above. We have that $0 \leq |\phi| \leq 1$. Therefore,
\begin{align*}
(\sigma_1^{-2} \phi(\boldsymbol\beta_1) - \sigma_2^{-2} \phi(\boldsymbol\beta_1))^2 &\leq (\sigma_1^{-2} - \sigma_2^{-2})^2 \cr
&=(\sigma_1^{-1/2}\sigma_1^{-3/2} - \sigma_2^{-1/2}\sigma_1^{-3/2} + \sigma_2^{-1/2}\sigma_1^{-3/2} - \sigma_2^{-1/2}\sigma_2^{-3/2})^2 \cr
&\leq 2 \sigma_1^{-3} (\sigma_1^{-1/2} - \sigma_2^{-1/2})^2 + 2 \sigma_2^{-1}(\sigma_1^{-3/2} - \sigma_2^{-3/2})^2 \cr
&\leq 2 \sigma_1^{-3} (\sigma_1^{-1/2} - \sigma_2^{-1/2})^2 + 4 \sigma_2^{-1} \sigma_1^{-2} (\sigma_1^{-1/2} - \sigma_2^{-1/2})^2 + 8 \sigma_2^{-1} \sigma_1^{-2}(\sigma_1^{-1/2} - \sigma_2^{-1/2})^2  \cr
&\quad + 8 \sigma_2^{-2} \sigma_1^{-1}(\sigma_1^{-1/2} - \sigma_2^{-1/2})^2 \quad \text{(repeated argument)} \cr
&\leq 12 (\sigma_1^{-1/2} - \sigma_2^{-1/2})^2(\sigma_1^{-3}  + \sigma_2^{-1} \sigma_1^{-2} + \sigma_2^{-2} \sigma_1^{-1}) \cr
&\leq 12 (\sigma_1^{-1/2} - \sigma_2^{-1/2})^2(\sigma_1^{-3}  + (1/2)\sigma_2^{-2} + (1/2) \sigma_1^{-4} + (1/2)\sigma_2^{-4} + (1/2) \sigma_1^{-2}) \cr
&\leq 24 (\sigma_1^{-1/2} - \sigma_2^{-1/2})^2(\sigma_1^{-4}  + \sigma_2^{-4}).
\end{align*}
It can be shown that $(\sigma_1^{-1/2} - \sigma_2^{-1/2})^2 \leq \|s_1 - s_2\|^2$ and that there exist positive constants $b_9$ and $b_{10}$ such that $\sigma_1^{-4}  + \sigma_2^{-4} \leq b_9(b_{10} + \|s_1\|^4 + \|s_2\|^4) \leq b_9 b_{10}(1 + \|s_1\|^4 + \|s_2\|^4) \leq b_9 b_{10}(1 + \|s_1\|^4 + \|s_2\|^4)^2$.

We now look at the second term in the squared brackets above:
\[
 (\sigma_2^{-2} \phi(\boldsymbol\beta_1) - \sigma_2^{-2} \phi(\boldsymbol\beta_2))^2 = \sigma_2^{-4} (\phi(\boldsymbol\beta_1) - \phi(\boldsymbol\beta_2))^2.
\]
It can be verified that the derivative of the function $\phi$ is given by a function, which is bounded in absolute value by $1$, to which we multiply $\mathbf{x}$. Therefore,
\[
 (\phi(\boldsymbol\beta_1) - \phi(\boldsymbol\beta_2))^2 \leq \max |x_i|^2 \|\boldsymbol\beta_1 - \boldsymbol\beta_2\|^2 \leq \|\mathbf{x}\|^2  \|\boldsymbol\beta_1 - \boldsymbol\beta_2\|^2.
\]
Also,
\[
 \sigma_2^{-4} \leq b_9 b_{10}(1 + \|s_1\|^4 + \|s_2\|^4)^2.
\]

We thus have an upper bound which is given by
\begin{align*}
 & \left\|-\sigma_1^{-2} \frac{\gamma + 1}{\gamma} \phi(\boldsymbol\beta_1) \, \mathbf{x} \mathbf{x}^T + -\sigma_2^{-2} \frac{\gamma + 1}{\gamma} \phi(\boldsymbol\beta_2) \, \mathbf{x} \mathbf{x}^T\right\|^2 \cr
 &\qquad \leq 48\left(\frac{\gamma + 1}{\gamma}\right)^2 \|\mathbf{x} \mathbf{x}^T\|^2(1 + \|\mathbf{x}\|^2) b_9 b_{10}(1 + \|s_1\|^4 + \|s_2\|^4)^2 \|\boldsymbol\theta_1 - \boldsymbol\theta_2\|^2.
\end{align*}

All the other bounds are of the same form.

We now show that we are able to obtain a bound of the form:
   \[
 \|l''(\mathbf{z}, \boldsymbol\theta_1)\|^2 \leq H^2(\mathbf{z}) [1 + \|\boldsymbol\theta_1\|^{b_6 + 1}]^2,
\]
for the part of the matrix associated with Block $A$. The bounds associated to the other blocks can be obtained similarly and they can all be put together to obtain the desired bound. We have
\begin{align*}
 & \left\|\sigma_1^{-2} \frac{\gamma + 1}{\gamma} \phi(\boldsymbol\beta_1) \, \mathbf{x} \mathbf{x}^T \right\|^2 \leq \left(\frac{\gamma + 1}{\gamma}\right)^2 \sigma_1^{-4} \|\mathbf{x} \mathbf{x}^T\|^2 \leq \left(\frac{\gamma + 1}{\gamma}\right)^2 b_9 b_{10}(1 + \|s_1\|^4)^2 \|\mathbf{x} \mathbf{x}^T\|^2.
\end{align*}

All the other bounds are of the same form.

We therefore have that $H^2(\mathbf{z}) = 48\left(\frac{\gamma + 1}{\gamma}\right)^2 \|\mathbf{x} \mathbf{x}^T\|^2(1 + \|\mathbf{x}\|^2) b_9 b_{10}$. We have that $\E[H(\mathbf{Z})] < \infty$ because we assumed that $\E\|\mathbf{X}\|^{4(p + 2)} < \infty$.

\textbf{A9}: \textit{We assume the expectations
\[
 I(\boldsymbol\theta) := \E[l'(\mathbf{Z}, \boldsymbol\theta) l'(\mathbf{Z}, \boldsymbol\theta)^T]
\]
and
\[
 M(\boldsymbol\theta) := -\E[l''(\mathbf{Z}, \boldsymbol\theta)]
\]
to exist and to be positive-definite matrices in a neighbourhood of $\boldsymbol\theta = (\boldsymbol\beta_0, \sigma^*)$.
}

We show that the assumption holds for $M(\boldsymbol\theta)$. The proof is similar for $ I(\boldsymbol\theta)$. We proceed by using Sylvester's criterion: a matrix is positive definite if and only if all the following matrices have a positive determinant: the upper left 1-by-1 corner, the upper left 2-by-2 corner, ..., the matrix itself. Recall that we assume that $\E[\mathbf{X} \mathbf{X}^T]$ is a positive-definite matrix.

We now show that the matrix defined by $-\E[A]$ has a positive determinant. The proof that the same is true for the smaller matrices is similar. We have
\begin{align*}
 -\E[A] &= \sigma^{-2} \frac{\gamma + 1}{\gamma} \int \mu_{\mathbf{X}}(\d \mathbf{x}) \, \mathbf{x} \mathbf{x}^T \int \frac{1 - \frac{(y - \mathbf{x}^T \boldsymbol\beta)^2}{\sigma^2 \gamma}}{\left(1 + \frac{(y - \mathbf{x}^T \boldsymbol\beta)^2}{\sigma^2 \gamma}\right)^2} \, p_0(y \mid \mathbf{x}) \, \d y \cr
 &= \sigma^{-2} \frac{\gamma + 1}{\gamma} \int \mu_{\mathbf{X}}(\d \mathbf{x}) \, \mathbf{x} \mathbf{x}^T \int \frac{1 - \frac{(u - \mathbf{x}^T (\boldsymbol\beta - \boldsymbol\beta_0)/\sigma_0)^2}{(\sigma / \sigma_0)^2 \gamma}}{\left(1 + \frac{(u - \mathbf{x}^T (\boldsymbol\beta - \boldsymbol\beta_0)/\sigma_0)^2}{(\sigma / \sigma_0)^2 \gamma}\right)^2} \, g(u) \, \d u \cr
 &= \sigma^{-2} \frac{\gamma + 1}{\gamma}\E\left[\mathbf{X} \mathbf{X}^T \frac{1 - \frac{(U - \mathbf{X}^T (\boldsymbol\beta - \boldsymbol\beta_0)/\sigma_0)^2}{(\sigma / \sigma_0)^2 \gamma}}{\left(1 + \frac{(U - \mathbf{X}^T (\boldsymbol\beta - \boldsymbol\beta_0)/\sigma_0)^2}{(\sigma / \sigma_0)^2 \gamma}\right)^2}\right],
\end{align*}
using the change of variable $u = \sigma_0^{-1}(y - \mathbf{x}^T\boldsymbol\beta_0)$. We view the random variable inside the expectation as a sequence indexed by $k$ where the terms that vary are $\boldsymbol\beta$ and $\sigma$ with $(\boldsymbol\beta_k, \sigma_k) \rightarrow (\boldsymbol\beta_0, \sigma^*)$ as $k \rightarrow \infty$. We have that
\begin{align*}
\sigma_k^{-2} \frac{\gamma + 1}{\gamma}\E\left[\mathbf{X} \mathbf{X}^T \frac{1 - \frac{(U - \mathbf{X}^T (\boldsymbol\beta_k - \boldsymbol\beta_0)/\sigma_0)^2}{(\sigma_k / \sigma_0)^2 \gamma}}{\left(1 + \frac{(U - \mathbf{X}^T (\boldsymbol\beta_k - \boldsymbol\beta_0)/\sigma_0)^2}{(\sigma_k / \sigma_0)^2 \gamma}\right)^2}\right] &\rightarrow (\sigma^*)^{-2} \frac{\gamma + 1}{\gamma}\E\left[\mathbf{X} \mathbf{X}^T \frac{1 - \frac{U^2}{(\sigma^* / \sigma_0)^2 \gamma}}{\left(1 + \frac{U^2}{(\sigma^* / \sigma_0)^2 \gamma}\right)^2}\right] \cr
& = (\sigma^*)^{-2} \frac{\gamma + 1}{\gamma}\E\left[\mathbf{X} \mathbf{X}^T\right]\E\left[ \frac{1 - \frac{U^2}{(\sigma^* / \sigma_0)^2 \gamma}}{\left(1 + \frac{U^2}{(\sigma^* / \sigma_0)^2 \gamma}\right)^2}\right].
\end{align*}
This follows from Lebesgue’s dominated convergence theorem and the fact that
\[
 0 \leq \frac{1 - \frac{(U - \mathbf{X}^T (\boldsymbol\beta_k - \boldsymbol\beta_0)/\sigma_0)^2}{(\sigma_k / \sigma_0)^2 \gamma}}{\left(1 + \frac{(U - \mathbf{X}^T (\boldsymbol\beta_k - \boldsymbol\beta_0)/\sigma_0)^2}{(\sigma_k / \sigma_0)^2 \gamma}\right)^2} \leq 1.
\]

We thus have that the limiting matrix has a positive determinant because the determinant of $\E\left[\mathbf{X} \mathbf{X}^T\right]$ is positive, $(\sigma^*)^{-2}, \frac{\gamma + 1}{\gamma} > 0$ and
\[
\E\left[ \frac{1 - \frac{U^2}{(\sigma^* / \sigma_0)^2 \gamma}}{\left(1 + \frac{U^2}{(\sigma^* / \sigma_0)^2 \gamma}\right)^2}\right] > 0
\]
for any $\gamma$. Therefore, by taking the neighbourhood around $(\boldsymbol\beta_0, \sigma^*)$ small enough, we know that the determinant of $-\E[A]$ is positive.

There remains to prove that the determinant of $M(\boldsymbol\theta)$ is positive. We use a limiting argument again. We know that
\[
 -\E[A] \rightarrow (\sigma^*)^{-2} \frac{\gamma + 1}{\gamma}\E\left[\mathbf{X} \mathbf{X}^T\right]\E\left[ \frac{1 - \frac{U^2}{(\sigma^* / \sigma_0)^2 \gamma}}{\left(1 + \frac{U^2}{(\sigma^* / \sigma_0)^2 \gamma}\right)^2}\right],
\]
which has a positive determinant. We now show that
\[
 -\E[B] \rightarrow \mathbf{0},
\]
which implies that the determinant of $M(\boldsymbol\theta)$ is approximately equal to that of $-\E[A]$ times the determinant of $-\E[D]$, the latter expectation being of dimension one and converging to a positive constant (which will be shown as well). This will allow to conclude.

We have
\begin{align*}
 -\E[B] &= 2\sigma^{-2}\sigma_s \frac{\gamma + 1}{\gamma}\int \mu_{\mathbf{X}}(\d \mathbf{x}) \, \mathbf{x} \int \frac{\frac{y - \mathbf{x}^T \boldsymbol\beta}{\sigma}}{\left(1 + \frac{(y - \mathbf{x}^T \boldsymbol\beta)^2}{\sigma^2 \gamma}\right)^2} \, p_0(y \mid \mathbf{x}) \, \d y \cr
 &= 2\sigma^{-2}\sigma_s \frac{\gamma + 1}{\gamma}\int \mu_{\mathbf{X}}(\d \mathbf{x}) \, \mathbf{x} \int \frac{\frac{u - \mathbf{x}^T (\boldsymbol\beta - \boldsymbol\beta_0)/\sigma_0}{\sigma / \sigma_0}}{\left(1 + \frac{(u - \mathbf{x}^T (\boldsymbol\beta - \boldsymbol\beta_0)/\sigma_0)^2}{(\sigma / \sigma_0)^2 \gamma}\right)^2} \, g(u) \, \d u \cr
 &= 2\sigma^{-2}\sigma_s \frac{\gamma + 1}{\gamma}\E\left[\mathbf{X}  \, \frac{\frac{U - \mathbf{X}^T (\boldsymbol\beta - \boldsymbol\beta_0)/\sigma_0}{\sigma / \sigma_0}}{\left(1 + \frac{(U - \mathbf{X}^T (\boldsymbol\beta - \boldsymbol\beta_0)/\sigma_0)^2}{(\sigma / \sigma_0)^2 \gamma}\right)^2}\right],
\end{align*}
using the change of variable $u = \sigma_0^{-1}(y - \mathbf{x}^T\boldsymbol\beta_0)$. As before, we view the random variable inside the expectation as a sequence indexed by $k$ where the terms that vary are $\boldsymbol\beta$ and $\sigma$ with $(\boldsymbol\beta_k, \sigma_k) \rightarrow (\boldsymbol\beta_0, \sigma^*)$ as $k \rightarrow \infty$. We have that
\begin{align*}
2\sigma_k^{-2}\sigma_{k,s} \frac{\gamma + 1}{\gamma}\E\left[\mathbf{X}  \, \frac{\frac{U - \mathbf{X}^T (\boldsymbol\beta_k - \boldsymbol\beta_0)/\sigma_0}{\sigma_k / \sigma_0}}{\left(1 + \frac{(U - \mathbf{X}^T (\boldsymbol\beta_k - \boldsymbol\beta_0)/\sigma_0)^2}{(\sigma_k / \sigma_0)^2 \gamma}\right)^2}\right] &\rightarrow 2(\sigma^*)^{-2}\sigma_{s}^* \frac{\gamma + 1}{\gamma}\E\left[\mathbf{X}  \, \frac{\frac{U}{\sigma^* / \sigma_0}}{\left(1 + \frac{U^2}{(\sigma^* / \sigma_0)^2 \gamma}\right)^2}\right] \cr
& = 2(\sigma^*)^{-2}\sigma_{s}^* \frac{\gamma + 1}{\gamma}\E[\mathbf{X}] \E\left[\frac{\frac{U}{\sigma^* / \sigma_0}}{\left(1 + \frac{U^2}{(\sigma^* / \sigma_0)^2 \gamma}\right)^2}\right] \cr
& = 2(\sigma^*)^{-2}\sigma_{s}^* \frac{\gamma + 1}{\gamma}\E[\mathbf{X}] \times 0.
\end{align*}
This follows from Lebesgue’s dominated convergence theorem and the fact that
\[
 0 \leq \frac{\frac{U - \mathbf{X}^T (\boldsymbol\beta_k - \boldsymbol\beta_0)/\sigma_0}{\sigma_k / \sigma_0}}{\left(1 + \frac{(U - \mathbf{X}^T (\boldsymbol\beta_k - \boldsymbol\beta_0)/\sigma_0)^2}{(\sigma_k / \sigma_0)^2 \gamma}\right)^2} \leq 1.
\]

A similar analysis allows to show that
\begin{align*}
 -\E[D] &\rightarrow (\sigma^*)^{-1} \sigma_{ss}^*\left(1- (\gamma + 1) \E\left[\frac{\frac{U^2}{(\sigma^* / \sigma_0)^2 \gamma}}{\left(1 + \frac{U^2}{(\sigma^* / \sigma_0)^2 \gamma}\right)}\right]\right) \cr
 &\qquad + (\sigma^*)^{-2}(\sigma_s^*)^2\left(2(\gamma + 1) \E\left[\frac{\frac{U^2}{(\sigma^* / \sigma_0)^2 \gamma}}{\left(1 + \frac{U^2}{(\sigma^* / \sigma_0)^2 \gamma}\right)^2}\right] + (\gamma + 1) \E\left[\frac{\frac{U^2}{(\sigma^* / \sigma_0)^2 \gamma}}{\left(1 + \frac{U^2}{(\sigma^* / \sigma_0)^2 \gamma}\right)}\right] - 1\right),
\end{align*}
which allows to conclude because for any $\gamma$
\[
 \left(1- (\gamma + 1) \E\left[\frac{\frac{U^2}{(\sigma^* / \sigma_0)^2 \gamma}}{\left(1 + \frac{U^2}{(\sigma^* / \sigma_0)^2 \gamma}\right)}\right]\right) > 0
\]
and
\[
 \left(2(\gamma + 1) \E\left[\frac{\frac{U^2}{(\sigma^* / \sigma_0)^2 \gamma}}{\left(1 + \frac{U^2}{(\sigma^* / \sigma_0)^2 \gamma}\right)^2}\right] + (\gamma + 1) \E\left[\frac{\frac{U^2}{(\sigma^* / \sigma_0)^2 \gamma}}{\left(1 + \frac{U^2}{(\sigma^* / \sigma_0)^2 \gamma}\right)}\right] - 1\right) > 0.
\]

\textbf{A10}: \textit{in the interior of $\re^{p+1}$, the loss function $L$ has continuous partial derivatives
\[
L^{(i,j)}(\boldsymbol\theta_1, \boldsymbol\theta_2) := \frac{\partial^{i + j}}{\partial \boldsymbol\theta_1^i \partial \boldsymbol\theta_2^j} L(\boldsymbol\theta_1, \boldsymbol\theta_2), \quad i, j = 1, 2.
\]
Moreover we assume with $c, b_7 > 0$ and for $i, j = 1, 2$,
\[
 \|L^{(i,j)}(\boldsymbol\theta_1, \boldsymbol\theta_2)\| \leq c (1 + \|\boldsymbol\theta_1\| + \|\boldsymbol\theta_2\|),
\]
for any $\boldsymbol\theta_1, \boldsymbol\theta_2$ in the interior of $\re^{p+1}$.
}

It can be readily verified that this assumption is verified when the loss function corresponds to the squared Euclidean norm.

\textbf{A11}: \textit{the prior measure has a density $\pi$ with respect to the Lebesgue measure on $\re^{p + 1}$, which is continuous on $\re^{p + 1}$ and fulfills for $b_8 > 0$,
\[
 0 < \pi(\boldsymbol\theta) < c(1+\|\boldsymbol\theta\|^{b_8}), \quad \boldsymbol\theta \in \re^{p + 1}.
\]
}

We assume that the prior density is strictly positive and that it is continuous. While verifying A5 we showed that it is bounded above by $6C(|s|+1)^3$. There exists a positive constant $c$ such that $6C(|s|+1)^3 < c(1 + \|s\|^3)$, which allows to verify A11. This concludes the proof of Result (c).
\end{proof}

\section{Details of the numerical experiment in \autoref{sec:robustness}}\label{sec:details_exp}

We use a HMC algorithm to sample from the posterior distribution. We thus apply a transformation on $\sigma$ to make it a variable on the real line. The original target density is such that:
\[
 \pi(\boldsymbol\beta, \sigma \mid \mathbf{y}) \propto \frac{1}{\sigma} \prod_{i = 1}^n \frac{1}{\sigma}f\left(\frac{y_i - \mathbf{x}_i^T \boldsymbol\beta}{\sigma}\right).
\]
We define $\nu := \log \sigma$, and thus,
\[
 \pi(\boldsymbol\beta, \nu \mid \mathbf{y}) \propto \frac{1}{\ee^{\nu n}} \prod_{i = 1}^n f\left(\frac{y_i - \mathbf{x}_i^T \boldsymbol\beta}{\ee^\nu}\right).
\]
The log density is such that (if we forget about the constants):
\[
 \log \pi(\boldsymbol\beta, \nu \mid \mathbf{y}) = -n \nu + \sum_{i = 1}^n \log f\left(\frac{y_i - \mathbf{x}_i^T \boldsymbol\beta}{\ee^\nu}\right).
\]
The gradient is such that:
\[
 \frac{\partial}{\partial \boldsymbol\beta} \log \pi(\boldsymbol\beta, \nu \mid \mathbf{y}) = \ee^{-\nu} \, \frac{\gamma + 1}{\gamma} \sum_{i = 1}^n \left(1 + \frac{(y_i - \mathbf{x}_i^T \boldsymbol\beta)^2}{\ee^{2 \nu} \gamma}\right)^{-1} \frac{y_i - \mathbf{x}_i^T \boldsymbol\beta}{\ee^{\nu}}  \, \mathbf{x}_i,
\]
\[
 \frac{\partial}{\partial \nu} \log \pi(\boldsymbol\beta, \nu \mid \mathbf{y}) =-n + \frac{\gamma + 1}{\gamma} \sum_{i = 1}^n  \left(1 + \frac{(y_i - \mathbf{x}_i^T \boldsymbol\beta)^2}{\ee^{2 \nu} \gamma}\right)^{-1} \frac{(y_i - \mathbf{x}_i^T \boldsymbol\beta)^2}{\ee^{2 \nu}}.
\]

We now perform the same calculations for the limiting posterior from which we sample. The original target density is such that:
\[
 \pi(\boldsymbol\beta, \sigma \mid \mathbf{y}_{\text{O}^\mathsf{c}}) \propto \frac{1}{\sigma} \, \sigma^{|\text{O}|\gamma} \prod_{i \in \text{O}^\mathsf{c}} \frac{1}{\sigma}f\left(\frac{y_i - \mathbf{x}_i^T \boldsymbol\beta}{\sigma}\right).
\]
After the change of variable $\nu = \log \sigma$, we have
\[
 \pi(\boldsymbol\beta, \nu \mid \mathbf{y}_{\text{O}^\mathsf{c}}) \propto \frac{1}{\ee^{\nu (|\text{O}^\mathsf{c}| - |\text{O}|\gamma)}} \prod_{i \in \text{O}^\mathsf{c}} f\left(\frac{y_i - \mathbf{x}_i^T \boldsymbol\beta}{\ee^\nu}\right).
\]
The log density is such that (if we forget about the constants):
\[
 \log \pi(\boldsymbol\beta, \nu \mid \mathbf{y}_{\text{O}^\mathsf{c}}) = -\nu (|\text{O}^\mathsf{c}| - |\text{O}|\gamma) + \sum_{i \in \text{O}^\mathsf{c}} \log f\left(\frac{y_i - \mathbf{x}_i^T \boldsymbol\beta}{\ee^\nu}\right).
\]
The gradient is such that:
\[
 \frac{\partial}{\partial \boldsymbol\beta} \log \pi(\boldsymbol\beta, \nu \mid \mathbf{y}_{\text{O}^\mathsf{c}}) = \ee^{-\nu} \, \frac{\gamma + 1}{\gamma} \sum_{i \in \text{O}^\mathsf{c}} \left(1 + \frac{(y_i - \mathbf{x}_i^T \boldsymbol\beta)^2}{\ee^{2 \nu} \gamma}\right)^{-1} \frac{y_i - \mathbf{x}_i^T \boldsymbol\beta}{\ee^{\nu}}  \, \mathbf{x}_i,
\]
\[
 \frac{\partial}{\partial \nu} \log \pi(\boldsymbol\beta, \nu \mid \mathbf{y}_{\text{O}^\mathsf{c}}) =-(|\text{O}^\mathsf{c}| - |\text{O}|\gamma) + \frac{\gamma + 1}{\gamma} \sum_{i \in \text{O}^\mathsf{c}}  \left(1 + \frac{(y_i - \mathbf{x}_i^T \boldsymbol\beta)^2}{\ee^{2 \nu} \gamma}\right)^{-1} \frac{(y_i - \mathbf{x}_i^T \boldsymbol\beta)^2}{\ee^{2 \nu}}.
\]

In \autoref{sec:robustness}, we also present numerical results for the normal linear regression. We use that the posterior mean of $\boldsymbol\beta$ is $(\mathbf{X}^T \mathbf{X})^{-1}  \mathbf{X}^T \mathbf{y}$ and that the posterior covariance matrix of $\boldsymbol\beta$ is
\[
 \frac{\|\mathbf{y} - \hat{\mathbf{y}}\|^2}{n - p - 2} \, (\mathbf{X}^T \mathbf{X})^{-1},
\]
$\mathbf{X}$ being here the design matrix and $\hat{\mathbf{y}} := \mathbf{X} (\mathbf{X}^T \mathbf{X})^{-1}  \mathbf{X}^T \mathbf{y}$.

\end{document}